\documentclass[11pt,a4paper]{article}

\usepackage{latexsym,amssymb,amsfonts,amsmath,epsfig,tabularx}
\usepackage{lscape}
\usepackage{bbm}
\usepackage{dsfont}
\usepackage{pdfsync}
\usepackage{xfrac}
\usepackage{subcaption}

\usepackage{algpseudocode,algorithm} 
\floatname{algorithm}{Pseudocode} 

\algdef{SE}[DOWHILE]{Do}{doWhile}{\algorithmicdo}[1]{\algorithmicwhile\ #1}%

\usepackage[x11names]{xcolor}
\usepackage[colorlinks=true,linkcolor=SteelBlue3,citecolor=SteelBlue3,urlcolor=SteelBlue3]{hyperref}
\usepackage[a4paper]{geometry}
\usepackage{amsthm}
   \theoremstyle{plain}

   \newtheorem{lemma}{Lemma}
   \newtheorem{corollary}{Corollary}

   \theoremstyle{definition}

   \theoremstyle{theorem} 
   \newtheorem{remark}{Remark}

\definecolor{darkred}{RGB}{139,0,0}
\definecolor{darkgreen}{RGB}{0,100,0}
\definecolor{darkmagenta}{RGB}{139,0,139}

\graphicspath{{.}{./code/Figures/}{./graphs/}}

\def\bbN{\mathbb{N}}

\def\eps{\varepsilon}
\def\e{\varepsilon}

\newcommand{\coloneqq}{:=}

\newcommand{\mm}{m}

\newcommand{\setu}{{\mathfrak{u}}}
\newcommand{\setv}{{\mathfrak{v}}}
\newcommand{\setw}{{\mathfrak{w}}}

\newcommand{\bigO}{\mathcal{O}}
\newcommand{\cost}{\mathrm{cost}}

\newcommand{\setU}{{\mathcal{U}}}

\newcommand{\satop}[2]{\stackrel{\scriptstyle{#1}}{\scriptstyle{#2}}}

\newcommand{\rd}{{\mathrm{d}}}

\newcommand{\calA}{{\mathcal{A}}}

\newcommand{\calI}{{\mathcal{I}}}

\newcommand{\calM}{{\mathcal{M}}}

\newcommand{\calP}{{\mathcal{P}}}

\newcommand{\calU}{{\mathcal{U}}}
\newcommand{\bsz}{{\boldsymbol{z}}}
\newcommand{\bsi}{{\boldsymbol{i}}}

\newcommand{\bsk}{{\boldsymbol{k}}}
\newcommand{\bsl}{{\boldsymbol{\ell}}}

\newcommand{\bsx}{{\boldsymbol{x}}}

\newcommand{\bst}{{\boldsymbol{t}}}
\newcommand{\bsp}{{\boldsymbol{p}}}
\newcommand{\bsq}{{\boldsymbol{q}}}

\newcommand{\bsw}{{\boldsymbol{w}}}
\newcommand{\bszero}{{\boldsymbol{0}}}
\newcommand{\bsone}{{\boldsymbol{1}}}

\newcommand{\bstau}{{\boldsymbol{\tau}}}
\newcommand{\bsDelta}{{\boldsymbol{\Delta}}}

\newcommand{\bbE}{{\mathbb{E}}}
\newcommand{\bbR}{{\mathbb{R}}}

\newcommand{\mask}[1]{}

\newcommand{\Uext}[0]{\ensuremath{\calU_\mathrm{ext}}}

\newcommand{\supdim}[0]{\sigma^*}

\makeatletter
\newcommand{\vast}{\bBigg@{4}}
\newcommand{\Vast}{\bBigg@{5}}
\makeatother

\newcommand{\pd}[3]{\dfrac{\partial ^{#1} #2}{\partial #3}}

\newcommand{\bs}[1]{\boldsymbol{#1}}
\newcommand{\nrm}[2]{\ensuremath{\left\|#1\right\|_{#2}}}

\newcommand{\N}[0]{\mathbb{N}}
\newcommand{\R}[0]{\mathbb{R}}

\makeatletter
\algnewcommand{\Statenonumber}[1]{\Statex \hskip\ALG@thistlm #1}
\makeatother

\title{Efficient implementations of the Multivariate Decomposition Method for
Approximating Infinite-Variate Integrals}

\author{Alexander D. Gilbert
             \and Frances Y. Kuo
             \and Dirk Nuyens
             \and Grzegorz W. Wasilkowski}
             
\begin{document}

\maketitle



\renewcommand{\thefootnote}{\fnsymbol{footnote}}


\begin{abstract}
In this paper we focus on efficient implementations of the Multivariate
Decomposition Method (MDM) for approximating integrals of
$\infty$-variate functions. Such $\infty$-variate integrals occur for
example as expectations in uncertainty quantification. Starting with the
anchored decomposition $f = \sum_{\setu\subset\bbN} f_\setu$, where the
sum is over all finite subsets of $\bbN$ and each $f_\setu$ depends only
on the variables $x_j$ with $j\in\setu$, our MDM algorithm approximates
the integral of $f$ by first truncating the sum to some `active set' and
then approximating the integral of the remaining functions $f_\setu$
term-by-term using Smolyak or (randomized) quasi-Monte Carlo (QMC)
quadratures. The anchored decomposition allows us to compute $f_\setu$
explicitly by function evaluations of $f$. Given the specification of the
active set and theoretically derived parameters of the quadrature rules,
we exploit structures in both the formula for computing $f_\setu$ and the
quadrature rules to develop computationally efficient strategies to
implement the MDM in various scenarios. In particular, we avoid repeated
function evaluations at the same point. We provide numerical results for a
test function to demonstrate the effectiveness of the algorithm.
\end{abstract}

\section{Introduction}\label{sec:intro}

The \emph{Multivariate Decomposition Method} (MDM) is an algorithm for
approximating the integral of an $\infty$-variate function $f$ defined on
some domain $D^\bbN$ with $D\subseteq\bbR$, and this paper presents the
first results on the implementation of the MDM. The general idea of the
MDM, see \cite{GilW17, KNPSW17, PW14, Was13a, Was14a} (as well as
\cite{KSWW10a,PW11} under the name of \emph{Changing Dimension
Algorithm}), goes as follows. Assume that $f$ admits a decomposition
\begin{align} \label{eq:fdecomp}
 f(\bsx) \,=\, \sum_{\setu\subset\bbN} f_\setu(\bsx_\setu),
\end{align}
where the sum is taken over all finite subsets of
\[
  \N \,\coloneqq\, \{1,2,3,\ldots\},
\]
and where each function $f_\setu$ depends only on the variables in
$\bsx_\setu=(x_j)_{j\in\setu}$. With $\rho$ a given probability density
function on $D$ and $\rho_\setu(\bsx) \coloneqq \prod_{j\in\setu} \rho(x_j)$, we
define the integral of $f$ by
\begin{align} \label{eq:int}
  \calI(f) \,\coloneqq\, \sum_{\setu\subset\bbN} I_\setu(f_\setu),
 \qquad\mbox{with}\qquad
 I_\setu(f_\setu) \,\coloneqq\, \int_{D^{|\setu|}} f_\setu(\bsx_\setu)
\rho_\setu(\bsx_\setu) \,\rd \bsx_\setu,
\end{align}
and let $I_\emptyset(f_\emptyset) \coloneqq f_\emptyset$. The MDM
algorithm for approximating the integral is
\begin{equation} \label{eq:mdm}
  \calA(f) \,\coloneqq\, \sum_{\setu\in\calU} A_\setu(f_\setu),
\end{equation}
where $\calU$ is the ``active set''; for $\setu\ne\emptyset$, each
$A_\setu$ is a $|\setu|$-dimensional quadrature rule, and $A_\emptyset
(f_\emptyset) := f_\emptyset$.

The error of the MDM algorithm satisfies the trivial bound
\begin{equation} \label{eq:triangle}
  |\calI(f) - \calA(f)|
  \,\le\, \sum_{\setu\notin\calU} |I_\setu(f_\setu)|
  \,+\, \sum_{\setu\in\calU} |I_\setu(f_\setu) - A_\setu(f_\setu)|.
\end{equation}
Given $\eps>0$, the strategy is to first choose an active set $\calU$ such
that the first sum in \eqref{eq:triangle} is at most $\eps/2$, and then
specify the quadrature rules such that the second sum in
\eqref{eq:triangle} is also at most $\eps/2$, giving a total error of at
most $\eps$. It is known that the sets $\setu\in\setU$ have
cardinalities increasing very slowly with decreasing $\e$,
\[\max_{\setu\in\setU(\eps)}|\setu|\,=\,
O\left(\frac{\ln(1/\e)}{\ln(\ln(1/\e))}\right)\quad
\mbox{as\ }\e\to0,
\]
see, e.g., \cite{GilW17,PW11}.

We would need to impose additional conditions on the class of functions to
ensure that the sum \eqref{eq:fdecomp} is absolutely convergent, the
integral \eqref{eq:int} is well defined, and the quadrature rules in
\eqref{eq:mdm} converge appropriately to the corresponding integrals. The
precise details will depend on the mathematical setting within which we
choose to analyze the problem. We will outline some variants below, but
this is not the focus of the present paper.

The main focus of this paper is on the implementation of the MDM
algorithm, which involves the two following steps.
The first step is to construct the active set given an abstract 
definition of $\calU$ from the theory.
Then in the second step, supposing we are given an active set and the
choice of quadrature rules $A_\setu$, 
 we develop computationally efficient
strategies to evaluate \eqref{eq:mdm} in certain scenarios by exploiting
specific structures in the MDM algorithm and the quadrature rules of
choice.
Specifically,
\begin{itemize}
\item we assume a \emph{product and order dependent} (POD) structure
    in the definition of the active set $\calU$;
\item we utilize the \emph{anchored decomposition} of functions; and
\item we consider \emph{quasi-Monte Carlo methods} and \emph{Smolyak's
    methods} as two alternatives for the quadrature rules
    $A_\setu$.
\end{itemize}

In Section~\ref{sec:act} we explain the structure of our active set
$\calU$ and provide an efficient strategy to construct it. Once the active
set $\calU$ has been constructed, we need to evaluate the quadrature rules
$A_\setu(f_\setu)$ for each $\setu\in\calU$; this is formulated in
Section~\ref{sec:form}. In this paper we use the anchored decomposition
\cite{KSWW10b} of $f$ so that the terms $f_\setu$ can be computed
explicitly via
\begin{equation} \label{eq:explicit}
  f_\setu(\bsx_\setu)
  \,=\, \sum_{\setv\subseteq\setu} (-1)^{|\setu|-|\setv|} f(\bsx_\setv;\bszero),
\end{equation}
where $f(\bsx_\setv;\bszero)$ indicates that we evaluate the function at
$f(\bst)$ with components $t_j = x_j$ for $j\in\setv$ and $t_j = 0$ for
$j\notin\setv$. Throughout this paper, by a ``naive'' implementation
of the MDM algorithm, we mean an implementation which computes the sum in
\eqref{eq:mdm} term by term, with each $f_\setu$ evaluated using
\eqref{eq:explicit}.

We consider only linear algorithms $A_\setu$ as the quadrature rules and
our MDM algorithm can therefore be expressed as
\begin{equation} \label{eq:mdm2}
  \calA(f)
  \,=\, \sum_{\setu\in\calU} \sum_{\setv\subseteq\setu} (-1)^{|\setu|-|\setv|}
  A_\setu(f(\cdot_\setv;\bszero)).
\end{equation}
Notice inside the double sum in \eqref{eq:mdm2} that we would be applying
a $|\setu|$-dimensional quadrature rule to a function which depends only
on a subset $\setv\subseteq\setu$ of the variables. Moreover, the same
evaluations of $f$ could be repeated for different combinations of $\setu$
and $\setv$, while in practice the cost of evaluating $f$ could be quite
expensive. We will exploit structures in the quadrature rules to save on
repeated evaluations in~\eqref{eq:mdm2}.

In Section~\ref{sec:smo} we first consider \emph{Smolyak quadrature} to be
used as the quadrature rules $A_\setu$ (see, e.g., \cite{GG98, Smol63}).
Then in Section~\ref{sec:qmc} we consider instead an \emph{extensible
quasi-Monte Carlo} (\emph{QMC}) \emph{sequence} to be used for the
quadrature rules (see, e.g., \cite{CKN06,DKS13}). In both sections we
explain how to regroup the terms by making use of the recursive structure
and how to store some intermediate calculations for the specific
quadrature rules to evaluate \eqref{eq:mdm2} efficiently.

Section~\ref{sec:smolyak} considers two different approaches to implement
the Smolyak quadratures: the direct method and the combination technique.
In Section~\ref{sec:RQMC} we consider a \emph{randomized} quasi-Monte
Carlo sequence for the quadrature rules. This enables us to obtain an
unbiased result and a practical estimate of the quadrature error for the
MDM algorithm.

Each variant of our MDM algorithm involves three stages, as outlined
in the pseudocodes; a summary is given as follows:
\[
\begin{tabular}{|l|l|}
 \hline
 Pseudocodes 1 + 2A + 3A & Smolyak MDM -- direct implementation \\
 Pseudocodes 1 + 2A$'$ + 3A$'$ & Smolyak MDM -- combination technique \\
 Pseudocodes 1 + 2B + 3B & Extensible QMC MDM \\
 Pseudocodes 1 + 2B + 3B$'$ & Extensible randomized QMC MDM \\
 \hline
 \end{tabular}
\]

In Section~\ref{sec:T} we derive a computable expression for
estimating an infinite series that may appear in the definition of the
active set, which is another novel and significant contribution 
of this paper. Finally in Section~\ref{sec:num} we combine all ingredients
and follow the mathematical setting of \cite{KNPSW17} to construct the
active set and choose the quadrature rules. We then apply the MDM
algorithm to an example integrand that mimics the characteristics of the
integrands arising from some parametrized PDE problems (see, e.g.,
\cite{KSS12}).

\section{Constructing the active set} \label{sec:act}

Letting $w(\setu)$ be a measure of the ``significance'' of the subset
$\setu$, we assume that the mathematical analysis yields the definition of
an active set of the general form
\begin{align} \label{eq:active}
  \calU \,\coloneqq\, \{ \setu \subset\bbN \,:\, w(\setu) > T \},
\end{align}
where $T > 0$
is a ``threshold'' parameter that depends on the overall error demand
$\eps>0$ and possibly on all of $w(\setu)$. For example, $w(\setu)$ can be
related to the weight parameters from a weighted function space setting
(as in \cite{Was13a,Was14a,PW14,GilW17}), or it can be related to the
bounds on the norm of $f_\setu$ (as in \cite{KNPSW17}).

In this section we will treat $T$ and $w(\setu)$ as input parameters
(ignoring the mathematical details of where they come from), and focus on
the efficient implementation of the active set given these parameters.
Then in Section~\ref{sec:T} we will consider a special form of $T$
(arising from analysis) which requires numerical estimation of an infinite
series.

We assume $w(\emptyset) > T$ so that we always have
$\emptyset\in\calU$. Furthermore, we assume specifically for
$\setu\ne\emptyset$ that $w(\setu)$ takes the \emph{product and order
dependent} (POD) form (a structure that first appeared in \cite{KSS12}):
\begin{align} \label{eq:POD}
  w(\setu) \,\coloneqq\, \Omega_{|\setu|}\,\prod_{j\in\setu} \omega_j,
\end{align}
where $\omega_1\ge\omega_2\ge\cdots$ is a non-increasing sequence of
nonnegative real numbers controlling the ``product aspect", and $\Omega_1,
\Omega_2,\ldots$ is a second sequence of nonnegative real numbers
controlling the ``order dependent aspect", with the
restriction on $\Omega_\ell$
that its growth is controlled by $\omega_\ell$, i.e.,
$\Omega_{\ell + 1} \omega_{\ell+1} \leq \Omega_\ell$ for all $\ell
\in \N$.
This assumption is satisfied
in all practical cases that we are aware of.
Further, in the theoretical framework for the MDM (see, e.g., \cite{KNPSW17}), 
a sufficient condition for the infinite-dimensional integral to be well-defined is for the parameters 
$w(\setu)$ to be summable,
$\sum_{\setu \subset \N} w(\setu) < \infty$,
which will not hold unless the condition $\Omega_{\ell+1} \omega_{\ell+1} \leq \Omega_{\ell}$
holds (at least asymptotically in $\ell$).

With the active set defined by \eqref{eq:active} and \eqref{eq:POD},
we make a couple of obvious remarks:
\begin{enumerate}
\item If $\setv\in\calU$ then $\setu\in\calU$ for all sets $\setu$
    satisfying $w(\setu)\ge w(\setv)$.
\item If $\setu\notin\calU$ then $\setv\notin\calU$ for all sets
    $\setv$ satisfying $w(\setu)\ge w(\setv)$.
\end{enumerate}
We identify any finite set $\setu\subset\bbN$ with a vector containing the
elements of $\setu$ in increasing order, i.e., if $|\setu| = \ell$ then
\[
  \setu \,\coloneqq\, (u_1,u_2,\ldots,u_\ell), \quad u_1<u_2<\cdots<u_\ell.
\]
Then, due to our assumed POD structure in \eqref{eq:POD}, we note that
\begin{enumerate}
\setcounter{enumi}{2}
\item $w(\setu)\ge w(\setv)$ if $|\setu|=|\setv|$ and $u_i\le v_i$ for
    all $i = 1, \dotsc, \ell$.
\item $w(\{1, \ldots, \ell\}) \geq w(\{1, \ldots, \ell + 1\})$ for all
    $\ell \in \N$.
\item\label{itm:dwopen} For any $\setu\in\calU$, a subset of $\setu$ need \emph{not} be
    included in $\calU$.
\end{enumerate}
Note that if the opposite of Item~\ref{itm:dwopen} were true, i.e., every subset of
$\setu\in\calU$ also belongs to $\calU$, then the set $\calU$ is said to
be ``downward closed'' in some papers; we do not impose this condition.

\smallskip

Combining the above, we deduce the following simple lemma.

\begin{lemma} \label{lem:dim}
Assume that the active set $\calU$ is defined by \eqref{eq:active} and
\eqref{eq:POD}.
\begin{itemize}
\item \textnormal{(``Superposition dimension'')} Let $\supdim$ be the
    largest possible value of $\ell$ for which $(1,2,\ldots,
    \ell)\in\calU$, i.e., $w(\{1,2,\ldots,\ell\})>T$. Then for all
    $\setu\in\calU$ we have $|\setu|\le \supdim$.

\item \textnormal{(``Truncation dimension for sets of order $\ell$'')}
    For any $\ell=1,\ldots,\supdim$, let $\tau_\ell$ be the largest possible
    value of $j\ge \ell$ for which $(1,2,\ldots,\ell-1,j)\in\calU$,
    That is, $w(\{1,2,\ldots,\ell-1,j\})>T$. Then for all
    $\setu\in\calU$ with $|\setu|= \ell$, we have $u_\ell\le
    \tau_\ell$; and consequently, $i\le u_i\le \tau_\ell - \ell + i$
    for all $1\le i\le \ell$.

\item \textnormal{(``Truncation dimension'')} Let $\tau^*$ be the
    largest possible value of $j$ for which $j\in\setu\in\calU$, i.e.,
    $\tau^* = \max_{\setu\in\calU} \max_{j\in\setu} j$. Then $\tau^* =
    \max_{1\le\ell\le \supdim} \tau_\ell$.
\end{itemize}
\end{lemma}

\begin{proof}
For the first point, suppose on the contrary that $\setu =
(u_1,\ldots,u_{\supdim+1})\in\calU$. Then letting $\setv \coloneqq
(1,\ldots,\supdim+1)$ we have $w(\setv)\ge w(\setu) > T$, which indicates that
$\setv\in\calU$, contradicting the definition of $\supdim$.

To demonstrate the second point, suppose on the contrary that $\setu =
(u_1,\ldots,u_\ell)\in\calU$ with $u_\ell > \tau_\ell$. Then we have
$\setv = (1,\ldots,\ell-1,u_\ell)$ with $w(\setv)\ge w(\setu) > T$, which
indicates that $\setv\in\calU$, but this contradicts the definition of
$\tau_\ell$. The bound on $u_i$ then follows easily.

The third point is straightforward.
\end{proof}

We construct the active set as outlined in Pseudocode~1. 
The algorithm adds the qualifying sets to the collection in the order of
increasing cardinality. For each $\ell\ge 1$, starting from the set
$(1,2,\ldots,\ell)$, the algorithm incrementally generates and checks sets
to be added to the collection. The algorithm terminates when it reaches a
value of $\ell$ for which $(1,2,\ldots,\ell)\notin\calU$, i.e.,
$w(\{1,2,\ldots,\ell\}) \le T$.

\begin{algorithm}[!h]
 \floatname{algorithm}{Pseudocode 1} 
 \caption{(Constructing the active set)} 
\begin{algorithmic}[1]
 \State Add $\emptyset$ to $\calU$
 \For{{\bf $\ell$ from $1$ to $\ell_{\rm threshold}$}}
  \Comment{$\ell_{\rm threshold}$ is a computational threshold}
    \State $\setu \leftarrow (1,2,\ldots,\ell)$
    \State $i \leftarrow \ell$
    \Comment{$i$ is the index for the next increment}
    \Loop
      \If{$w(\setu) > T$}
        \State $i \leftarrow \ell$
        \State Add $\setu$ to $\calU$
    \Comment{add $\setu$ to the active set}
      \Else
        \State $i \leftarrow i - 1$
       \EndIf
      \State \textbf{break the inner loop if $i=0$}
      \Comment{move to next cardinality}
     \For{{\bf $j$ from $i$ to $\ell$}}
        \Comment{increment $\setu$ from $u_i$}
        \State $u_j \leftarrow u_i + j - i + 1$
      \EndFor
    \EndLoop
    \State \textbf{break the outer loop if} no sets of size $\ell$
    found
    \Comment{terminate}
  \EndFor
\end{algorithmic}
\end{algorithm}

The assumptions on the structure of $w(\setu)$ and properties 1--5 above 
ensure that this stopping criteria is valid, and hence that 
Pseudocode~1 does indeed construct the active set \eqref{eq:active}.
In particular, property 3 implies $w(\setu) \leq w(\{1, 2, \ldots, \ell\})$ for all
sets with $|\setu| = \ell$, and then 
Property 4 implies $w(\setu) \leq w(\{1, 2, \ldots, \ell\})$
for all sets with $|\setu| \geq \ell$. Thus, if $\{1, 2, \ldots, \ell\} \notin \calU$
then no set with cardinality $\ell$ or higher is in $\calU$.

We recommend storing the active set $\calU$ as an array of hash tables,
with one table for each cardinality, since in the next section we will
have to iterate over all subsets $\setv \subseteq \setu \in \calU$ and be
able to update a table stored with each such $\setv$.


\begin{remark} \label{rem1}
The paper \cite{GilW17} provides an efficient algorithm to construct
the optimal active set $\calU^{\rm opt}(\e)$, i.e., an active set that has
the smallest cardinality among all active sets $\calU(\e)$ with the same
error demand $\eps$. The construction principle is based on sorting so is
quite different to Pseudocode~1, and it works only for parameters of
product form. Once the active set is constructed, the remaining steps for
implementing the MDM algorithm will be the same as we discuss below.
\end{remark}

\section{Formulating the MDM algorithm}
\label{sec:form}

In this section we outline how to formulate the MDM algorithm
\eqref{eq:mdm2} in a way that is specific to the quadrature rules used, so
that the implementation can be as efficient as possible. We do this by
exploiting the structure in the anchored decomposition
\eqref{eq:explicit}, and also in the quadrature rules, which will be
Smolyak's methods (also known as sparse grid methods) and quasi-Monte
Carlo rules.

In each case we treat the parameters of the quadrature rules (i.e.,
the number of quadrature points or levels) as input, and focus on the
efficient implementation of the MDM given these parameters. Specific
choices of parameters for a test integrand following the theoretical
analysis in \cite{KNPSW17} are given in Section~\ref{sec:num}.

Recall from \eqref{eq:mdm2} that the MDM algorithm using the anchored
decomposition is given by
\[
\calA(f) \,=\, \sum_{\setu \in \calU} \sum_{\setv \subseteq \setu}
(-1)^{|\setu| - |\setv|} A_\setu(f(\cdot_\setv; \bs0))\,.
\]
Clearly there will be subsets $\setv$ that will occur many times over, so
implementing the MDM in this way could be severely inefficient, because it
would evaluate the same functions $f(\cdot_\setv; \bs0)$ at the same
quadrature points over and over again. The goal of this section is to
detail how to implement the quadrature approximations in such a way that
each function $f(\cdot_\setv; \bs0)$ is evaluated at each quadrature point
\emph{once only}.

The first step is to introduce the \emph{extended active set}:
\[
  \calU_\mathrm{ext} \coloneqq \{\setv \subset \N : \setv \subseteq
  \setu \text{ for } \setu \in \calU\}\,,
\]
that is, it includes all subsets of the sets in the active set. Then we
can swap the sums above to give
\begin{align} \label{eq:mdm_ext}
 \calA(f) &\,=\, \sum_{\setv \in \Uext}
 \sum_{\satop{\setu \in \calU}{\setu \supseteq \setv}} (-1)^{|\setu| - |\setv|}
 A_\setu(f(\cdot_\setv; \bs0)) \nonumber\\
 &\,=\, c_\emptyset\,f(\bszero)
 + \sum_{\emptyset \neq \setv \in \Uext}
 \sum_{\satop{\setu \in \calU}{\setu \supseteq \setv}} (-1)^{|\setu| - |\setv|}
 A_\setu(f(\cdot_\setv; \bs0))\,,
\end{align}
where we separated out the $\setv = \emptyset$ terms, with
\begin{align*}
 c_\emptyset \,:=\, \sum_{\setu \in \calU} (-1)^{|\setu|}.
\end{align*}
After constructing the active set $\calU$, we go through it again to
construct the extended active set $\Uext$, and at the same time store
information regarding the superset structure of each element in $\Uext$.
We would like to store just enough details so that for each $\setv \in
\Uext$ we can compute the approximation $\sum_{\setu \in \calU \,:\, \setu
\supseteq \setv} (-1)^{|\setu| - |\setv|} A_\setu(f(\cdot_\setv; \bs0))$
without the need to access the supersets of~$\setv$. Specific details on
how this is done will depend on the quadrature rule used.

\subsection{Quadrature rules based on Smolyak's method} \label{sec:smo}

For a nonempty set $\setu \subset \bbN$ and integer $m\ge 1$, Smolyak's
method (see, e.g., \cite{GG98,Smol63,WW95}) applied to a function 
$g_\setu$ of the variables $\bsx_\setu$ takes
the form
\begin{equation} \label{eq:smolyak}
  Q_{\setu,\mm} (g_\setu)
  \,\coloneqq\,
  \sum_{\satop{\bsi_\setu\in\bbN^{|\setu|}}{|\bsi_\setu|\le |\setu| + \mm - 1}}
  \bigotimes_{j\in\setu} (U_{i_j} - U_{i_j-1}) (g_\setu),
\end{equation}
where $|\bsi_\setu| \coloneqq \sum_{j\in\setu} i_j$, $\mm \ge 1$, and
$\{U_i\}_{i\ge 1}$ is a sequence of one-dimensional quadrature rules, not
necessarily nested, with $U_0 \coloneqq0$ denoting the zero algorithm.
Furthermore we assume that constant functions are integrated exactly, so
that $U_i(1) = 1$ for $i \geq 1$.

For a nonempty subset $\setv\subseteq\setu$, suppose now that the
function $g_\setv$ depends only on the variables $\bsx_\setv$. Then we
have
\begin{align} \label{eq:proj}
  Q_{\setu,\mm} (g_\setv)
  &\,=\, \sum_{\satop{\bsi_\setu\in\bbN^{|\setu|}}{|\bsi_\setu|\le |\setu| + \mm - 1}}
  \bigg(\bigotimes_{j\in\setv} (U_{i_j} - U_{i_j-1})(g_\setv) \bigg)
  \bigg(\bigotimes_{j\in\setu\setminus\setv} (U_{i_j} - U_{i_j-1})(1) \bigg) \nonumber\\
  &\,=\, \sum_{\satop{\bsi_\setu\in\bbN^{|\setu|}}{|\bsi_\setu|\le |\setu| + \mm - 1,\,\bsi_{\setu\setminus\setv}=\bsone}}
  \bigotimes_{j\in\setv} (U_{i_j} - U_{i_j-1})(g_\setv) \nonumber\\
  &\,=\, \sum_{\satop{\bsi_\setv\in\bbN^{|\setv|}}{|\bsi_\setv|\le |\setv| + \mm - 1}}
  \bigotimes_{j\in\setv} (U_{i_j} - U_{i_j-1})(g_\setv) 
  \,=\, Q_{\setv,\mm}(g_\setv).
\end{align}
In the second equality above we used the assumption that the
one-dimensional quadrature rules integrate the constant functions exactly
and thus $(U_{i_j} - U_{i_j-1})(1)$ is~$1$ if $i_j=1$ and is~$0$
otherwise. The above derivation \eqref{eq:proj} indicates how a Smolyak
quadrature rule is projected down when it is applied to a lower
dimensional function. This property is important in our efficient
evaluation of \eqref{eq:mdm_ext}.

In \eqref{eq:mdm_ext} we take
\[
  A_\setu \,\equiv\, Q_{\setu,\mm_\setu},
\]
where the level $\mm_\setu$ determines the number of quadrature points $n_\setu$
used by $Q_{\setu, m_\setu}$.
The exact relationship between $m_\setu$ and $n_\setu$ will depend on the choice of the
one-dimensional quadrature rules $\{U_i\}$.

Here we treat the levels $m_\setu$, hence also the number of points $n_\setu$, as input parameters
to our MDM algorithm. Then we define the maximum level to occur as
\begin{align*}
 \mm_{\max} \,\coloneqq\, \max \{\mm_\setu : \emptyset\ne\setu\in\calU\}.
\end{align*}
For Smolyak grids based on one-dimensional rules $\{U_i\}$ that each use
$\bigO(2^i)$ points (e.g., trapezoidal rules) the value of $\mm_\setu$ is roughly the logarithm of $n_\setu$
(see, e.g., \cite{GG98}).
Hence in practice we observe that
$\mm_{\max}$ is relatively small, e.g., $\mm_{\max}\approx 25$.

Using \eqref{eq:proj} we can rewrite \eqref{eq:mdm_ext} as follows (note
the change from $Q_{\setu,\mm_\setu}$ to $Q_{\setv,\mm_\setu}$):
\begin{align} \label{eq:smo1}
  \calA^\mathrm{S}(f)
  &\,=\,
   c_\emptyset\,f(\bszero) +
  \sum_{\emptyset \neq\setv \in \Uext}  \sum_{\satop{\setu\in\calU}{\setu \supseteq \setv}}  (-1)^{|\setu|-|\setv|}\,
  Q_{\setu,\mm_\setu}(f(\cdot_\setv;\bszero)) \nonumber\\
  &\,=\,  c_\emptyset\,f(\bszero) +
  \sum_{\emptyset \ne\setv\in\calU_{\rm ext}}
  \sum_{\satop{\setu\in\calU}{\setu\supseteq\setv}} (-1)^{|\setu|-|\setv|}\,
  Q_{\setv,\mm_\setu}(f(\cdot_\setv;\bszero)) \nonumber\\
  &\,=\, c_\emptyset\,f(\bszero) + \sum_{\emptyset\ne\setv\in\calU_{\rm ext}}
  \sum_{\satop{m = 1}{c(\setv,m)\ne0}}^{\mm_{\max}} c(\setv,m)\,
  Q_{\setv, m}(f(\cdot_\setv;\bszero))\,,
\end{align}
where for $\setv\ne\emptyset$ and $m=1,\ldots,m_{\max}$ we define
\begin{align} \label{eq:c-smol}
 c(\setv, m) \,\coloneqq\,
 \sum_{\satop{\setu \in \calU,\,\setu \supseteq \setv}{\mm_\setu = m}}
 (-1)^{|\setu| - |\setv|}\,.
\end{align}

\begin{algorithm}[!t]
 \floatname{algorithm}{Pseudocode 2A} 
 \caption{(Constructing the extended active set for Smolyak)}
 \begin{algorithmic}[1]
 \State Initialize $\Uext \leftarrow \calU$
 \Comment start from the active set
 \State Initialize $c_\emptyset \leftarrow 1$
 \Comment for $\setu=\emptyset$
 \For{$\emptyset\ne\setu\in\calU$ with $|\setu|$ from $1$ to $\supdim$}
 \Comment traverse in increasing cardinality
 \State Calculate $m_\setu$
 \Comment formula for $m_\setu$ is given from theory
 \State Update $c_\emptyset \leftarrow c_\emptyset + (-1)^{|\setu|}$
 \State Initialize $c(\setu, m) \leftarrow 0$ for $m$ from $1$ to $\mm_{\max}$
 \For{$\emptyset \neq\setv \subseteq \setu$}
 \Comment generate nonempty subsets
 \If{$\setv \not \in \Uext$}
 \Comment look up and add missing subset
 \State Add $\setv$ to $\Uext$
 \State Initialize $c(\setv, m) \leftarrow 0$ for $m$ from $1$ to $\mm_{\max}$
 \EndIf
 \State Update $c(\setv, m_\setu) \leftarrow c(\setv, m_\setu) + (-1)^{|\setu| - |\setv|}$
 \Comment update relevant entry
 \EndFor \EndFor
\end{algorithmic}
\end{algorithm}
\begin{algorithm}[!h]
 \floatname{algorithm}{Pseudocode~3A} 
 \caption{(Implementing the Smolyak MDM)}
\begin{algorithmic}[1]
 \State Initialize $\calA^{\rm S}(f) \leftarrow c_\emptyset\times f(\bszero)$
 \For{$\emptyset\ne\setv \in \Uext$}
 \For{$m$ from $1$ to $\mm_{\max}$}
 \If{$c(\setv,m)\ne 0$}
 \State Calculate $Q_{\setv,m}(f(\cdot_\setv;\bszero))$ using
   \eqref{eq:Sm_nonnest}--\eqref{eq:nonnest_weights} or
   \eqref{eq:Sm_nest}--\eqref{eq:nest_weights}
 \State Update $\calA^{\rm S}(f) \leftarrow \calA^{\rm
 S}(f) + c(\setv,m)\times Q_{\setv,m}(f(\cdot_\setv;\bszero))$
 \EndIf \EndFor\EndFor \\
 \Return $\calA^{\rm S}(f)$
\end{algorithmic}
\end{algorithm}

The values of $c(\setv, m)$ can be computed and stored while we construct
the extended active set $\Uext$ as follows. We work through the sets in
the active set in order of increasing cardinality.
For each nonempty set $\setu \in \calU$ with required level
$m_\setu$, we generate all nonempty subsets $\setv\subseteq\setu$, add the
missing subsets to $\Uext$, and update $c(\setv, m_\setu)$ as we go. This
procedure is given in Pseudocode~2A.

This formulation \eqref{eq:smo1}--\eqref{eq:c-smol} allows us to compute
the $Q_{\setv,m}(f(\cdot_\setv;\bszero))$ for different supersets $\setu$
with the same value of $\mm_\setu$ only once. If the Smolyak MDM algorithm
is implemented in this way then there is no need to access the superset
structure. Obviously, if $c(\setv, m) = 0$ then we do not perform the
quadrature approximation.

Note that in practice calculating the number of Smolyak levels $m_\setu$
(or the number of QMC points in the next subsection) normally requires
knowledge of the entire active set $\calU$, see, e.g.,
\cite[Section~4.3]{KNPSW17}, hence we compute them when constructing
$\Uext$.

Note also that we do not need a separate data structure for $\Uext$: we can
simply extend $\calU$ to $\Uext$ since Step~9 in
Pseudocode~2A 
only adds subsets with lower cardinalities and would not interfere with
Step~3 since we iterate in increasing cardinality. As we explained in the
previous section, we store the active set $\calU$, and by extension the
extended active set $\Uext$, as an array of hash tables to easily retrieve
the $c(\setv,\cdot)$ table for each $\setv$.

A direct implementation of the MDM algorithm with Smolyak quadratures
is given in Pseudocode~3A. 
The different formulas \eqref{eq:Sm_nonnest}--\eqref{eq:nonnest_weights}
or \eqref{eq:Sm_nest}--\eqref{eq:nest_weights} for implementing the
Smolyak quadrature, which depend on whether we have a non-nested or nested
rule, will be discussed in Section~\ref{sec:smolyak}.

\subsection{Quadrature rules based on quasi-Monte Carlo methods}
\label{sec:qmc}

Here we assume for simplicity that $D = [0,1]$ and $\rho\equiv 1$.
A $d$-dimensional \emph{quasi-Monte Carlo} (QMC) rule with
$n$ points $\bst^{(i)} =
(t^{(i)}_1,t^{(i)}_2,\ldots,t^{(i)}_d)$, $i=0,\ldots,n-1$,
approximates the integral of a function $g$ by the equal-weight average
\begin{equation}
\label{eq:generic_qmc}
  \int_{[0,1]^d} g(\bsx)\,\rd\bsx
  \,\approx\,
  \frac{1}{n} \sum_{i=0}^{n-1} g(\bst^{(i)}).
\end{equation}
For more details on QMC methods we refer to \cite{Nie92,SJ94} and
\cite{DKS13}.

In \eqref{eq:mdm} each $A_\setu$ could be a different
$|\setu|$-dimensional QMC rule with $n_\setu$ points, but in that case we
would not be able to reuse any function evaluation.
Instead, we consider here an ``\emph{extensible quasi-Monte Carlo
sequence}''. By ``extensible'' we mean that we can take just the initial
dimensions of the initial points in the sequence. By ``quasi-Monte Carlo''
we mean that a quadrature rule based on the first $n$ points of this
sequence has equal quadrature weights $1/n$. We choose to use a QMC rule
with $\supdim$ dimensions instead of $\tau^*$, since $\tau^*$ can be really
large (e.g., $30000$) while $\supdim$ is rather small (e.g., $15$), and QMC
rules with fewer dimensions are of better quality.

Then for any nonempty set $\setu\in\calU$ and nonempty subset
$\setv\subseteq\setu$ we have
\begin{equation} \label{eq:A-qmc}
  A_\setu(f(\cdot_\setv;\bszero))
  \,=\, \frac{1}{n_\setu} \sum_{i=0}^{n_\setu-1}
  f\left(\bst^{(i)}_{\setu|\setv\to\setv};\bszero\right),
\end{equation}
where, loosely speaking, $\bst^{(i)}_{\setu|\setv\to\setv}$ indicates 
that we map the quadrature point $\bst^{(i)}$ to the variables $\bsx_\setu$
and then to $\bsx_\setv$, which is \emph{not} the same as mapping
$\bst^{(i)}$ directly to $\bsx_\setv$.
More explicitly, recalling that $\setu$ has ordered elements,
$\bst^{(i)}_{\setu | \setv \to \setv}$ denotes that we
take the first $|\setu|$-dimensions of $\bst^{(i)}$ and apply them to the variables
in $\bsx_\setu$, then retain only those components in
$\bsx_\setv$. 
The function $f$ is then evaluated by anchoring all
other components outside of $\setv$ to zero. 
%
%
%
Thus the algorithm \eqref{eq:mdm_ext} in this case is given by
\begin{align} \label{eq:qmc1}
  \calA^\mathrm{Q}(f)
  &\,=\,c_\emptyset\,f(\bszero) +
  \sum_{\emptyset \neq\setv\in\Uext}
  \sum_{\satop{\setu \in \calU}{\setu\supseteq\setv}} (-1)^{|\setu|-|\setv|}\,
  \left(
  \frac{1}{n_\setu} \sum_{i=0}^{n_\setu-1} f\left(\bst^{(i)}_{\setu|\setv\to\setv}
  ;\bszero\right)
  \right).
\end{align}

For example, 
take $\setu = (1,5,7)$ and $\setv = (1,7)$. We get
$\setu|\setv = (1,3)$ since the set $\setv$ originates from the position
$\setw = (1,3)$ of its superset $\setu$. We assign the quadrature point
$(t_1^{(i)},t_2^{(i)},t_3^{(i)})$ to the variables $(x_1,x_5,x_7)$. Then
the point $(t_1^{(i)},t_3^{(i)})$ is assigned to the variables
$(x_1,x_7)$, and hence we evaluate
$f(t_1^{(i)},0,0,0,0,0,t_3^{(i)},0,\ldots)$.

Note that the same set $\setv = (1,7)$ can originate from the position
$\setw = (1,3)$ of different supersets~$\setu$: for example, $\setu =
(1,6,7)$, $\setu = (1,4,7,13)$, and many others. We can make use of this
repetition to save on computational cost.

Let $\calM(\setv)$ denote the set of all different positions that a
nonempty set $\setv$ can originate from for all its supersets in the
active set:
\[
\calM(\setv) \,\coloneqq\, \left\{\setw \subseteq \{1, \ldots, \supdim\} \,:\,
\setw \,\equiv\, \setu|\setv
\text{ for some } \setu \in \calU \text{ with } \setu \supseteq \setv
\right\}.
\]
For simplicity and for convenience, we assume further that $n_\setu =
2^{m_\setu}$, with $0\le m_\setu\le m_{\max}$ (e.g., with $m_{\max}
\approx 25$). This allows us to rewrite each QMC approximation as a
sum of blocks of points (recall that the QMC points are extensible)
\begin{equation*}
\frac{1}{2^{m_\setu}}\sum_{i = 0}^{2^{m_\setu} - 1}f\left(
\bst^{(i)}_{\setu |\setv \rightarrow \setv}; \bs0\right)
\,=\,
\frac{1}{2^{m_\setu}}\sum_{m = 0}^{m_\setu}
\sum_{i = \lfloor 2^{m - 1} \rfloor}^{2^m - 1}
f\left(\bst^{(i)}_{\setu|\setv \rightarrow \setv}; \bs0\right)
\,,
\end{equation*}
where the floor function is used specifically to take care of the $m=0$
case.
Substituting this into \eqref{eq:qmc1} and introducing a sum over
$\calM(\setv)$, we have
\begin{align*}
  \calA^\mathrm{Q}(f)
  \,=\,c_\emptyset\,f(\bszero)
  \,+\, \sum_{\emptyset \neq \setv\in\Uext}
  \sum_{\setw \in \calM(\setv)}
  \sum_{\substack{\setu \in \calU\\\setu\supseteq\setv\\ \setu|\setv \equiv \setw }} (-1)^{|\setu|-|\setv|}\,
   \frac{1}{2^{m_\setu}}\sum_{m = 0}^{m_\setu}
   \sum_{i = \lfloor 2^{m - 1}\rfloor}^{2^m - 1}
   f\left(\bst^{(i)}_{\setw\rightarrow \setv}; \bs0\right)\,,
\end{align*}
where in the third sum we have added the restriction that $\setu|\setv$ is
equivalent to the position~$\setw$.
Collecting the sums
\begin{equation} \label{eq:s-qmc1}
 S_{\setv, \setw, m}(f) \,\coloneqq\,\sum_{i = \lfloor 2^{m - 1}\rfloor}^{2^m - 1}
 f\left(\bst^{(i)}_{\setw\rightarrow \setv}; \bs0\right)\,,
\end{equation}
we can then rewrite the QMC MDM \eqref{eq:qmc1} as
\begin{align}
\label{eq:qmc}
 \calA^\mathrm{Q}(f) \,=\, c_\emptyset\,f(\bszero)
 + \sum_{\emptyset\ne \setv \in \Uext}
 \sum_{\setw \in \calM(\setv)}
 \sum_{\satop{m=0}{c(\setv, \setw, m)\ne 0}}^{m_{\max}}
 c(\setv, \setw, m)
 \frac{S_{\setv, \setw, m}(f)}{2^{m_{\max}}} \,,
\end{align}

\begin{algorithm}[!h]
 \floatname{algorithm}{Pseudocode 2B} 
 \caption{(Constructing the extended active set for QMC)}
 \begin{algorithmic}[1]
 \State Initialize $\Uext \leftarrow \calU$
 \Comment start from the active set
 \State Initialize $c_\emptyset \leftarrow 1$
 \Comment for $\setu=\emptyset$
 \For{$\emptyset\ne\setu \in \calU$ with $|\setu|$ from $1$ to $\supdim$}
 \Comment traverse in increasing cardinality
 \State Calculate $m_\setu$
 \Comment formula for $m_\setu$ is given from theory
 \State Update $c_\emptyset \leftarrow c_\emptyset + (-1)^{|\setu|}$
 \State Initialize $\calM(\setu) \leftarrow \emptyset$
 \For {$\emptyset \neq \setv \subseteq \setu$}
 \Comment generate nonempty subsets
 \If{$\setv \not \in \Uext$}
 \Comment look up and add missing subset
 \State Add $\setv$ to $\Uext$
 \State Initialize $\calM(\setv) \leftarrow \emptyset$
 \EndIf
 \State Set $\setw \leftarrow \setu|\setv$
 \Comment identify the position where $\setv$ originates from $\setu$
 \If{$\setw \not \in \calM(\setv)$}
 \Comment look up and add missing position
 \State Add $\setw$ to $\calM(\setv)$
 \State Initialize $c(\setv,\setw, m) \leftarrow 0$ for $m$ from $1$ to $\mm_{\max}$
 \EndIf
 \For{$m$ from $0$ to $m_\setu$}
 \Comment update relevant entries
 \State Update $c(\setv,\setw,m) \leftarrow c(\setv,\setw,m) +
      (-1)^{|\setu| - |\setv|} \times 2^{m_{\max} - m_\setu}$
   \EndFor
 \EndFor \EndFor
\end{algorithmic}
\end{algorithm}
\begin{algorithm}[!h]
 \floatname{algorithm}{Pseudocode~3B} 
 \caption{(Implementing the QMC MDM)}
\begin{algorithmic}[1]
 \State Initialize $\calA^{\rm Q}(f) \leftarrow c_\emptyset\times f(\bszero)$
 \For{$\emptyset\ne\setv \in \Uext$}
 \For{$\setw \in \calM(\setv)$}
 \For{$m$ from $0$ to $\mm_{\max}$}
 \If{$c(\setv,\setw,m)\ne 0$}
 \State Calculate $S_{\setv,\setw,m}(f)$ using \eqref{eq:s-qmc1}
 \State Update $\calA^{\rm Q}(f) \leftarrow \calA^{\rm
 Q}(f) + c(\setv,\setw,m)\times S_{\setv,\setw,m}(f) / 2^{m_{\max}}$
 \EndIf \EndFor\EndFor\EndFor \\
 \Return $\calA^{\rm Q}(f)$
\end{algorithmic}%
\end{algorithm}%
\noindent
where for a nonempty set $\setv$, a position $\setw\in\calM(\setv)$, and
$m=0,\ldots,m_{\max}$ we define
\begin{align} \label{eq:c-qmc}
 c(\setv, \setw, m) \,\coloneqq\,
 \sum_{\satop{\setu\in\setU,\,\setu \supseteq \setv}
             {\setu|\setv \,\equiv\, \setw,\, m_\setu \geq m}} (-1)^{|\setu| - |\setv|}\,
  2^{m_{\max} - m_\setu}\,.
\end{align}
Note that we have chosen to multiply and divide by $2^{m_{\max}}$ to
ensure that each $c(\setv, \setw, m)$ is integer valued.

We can compute and store a list of positions $\calM(\setv)$ and the values
$c(\setv, \setw, m)$ when we construct the extended active set $\Uext$ by
extending the active set $\calU$, in a similar way to the Smolyak case in
the previous subsection. This is
presented in Pseudocode~2B. 
The new algorithm is more complicated due to the need to store the
positions $\calM(\setv)$.

The MDM implementation using the formulation
\eqref{eq:s-qmc1}--\eqref{eq:c-qmc} does not require access to any subsets
or supersets. For each nonempty $\setv\subseteq\Uext$ and each position
$\setw\in\calM(\setv)$ and for the different $m$, with $c(\setv,\setw, m)
\ne 0$, the sums $S_{\setv, \setw, m}(f)$ are over disjoint sets of QMC
points. In this way we will only evaluate each function
$f(\cdot_{\setw\to\setv}; \bs0)$ at each quadrature point once. An
implementation of the MDM algorithm with QMC quadratures is given in
Pseudocode~3B.

\section{Two implementations of Smolyak MDM} \label{sec:smolyak}

Here we compare two approaches to implement Smolyak quadrature in the
context of MDM: the direct implementation and the combination technique.

\subsection{Direct Smolyak implementation} \label{sec:smo_direct}

From a practical point of view, it is more useful to write Smolyak's
method as an explicit weighted quadrature rule as opposed to the tensor
product form \eqref{eq:smolyak}, see, e.g., \cite{GG98}. We summarize this
formulation below.

For each one-dimensional rule $U_i$, let $n_i$ denote the number of
quadrature points, $(w_{i, k})_{k = 0}^{n_i - 1}$ the quadrature weights,
and $(t_{i,k})_{k = 0}^{n_i - 1}$ the quadrature nodes. Here for
simplicity of notation we present the formula for a $d$-dimensional rule,
with $d\ge 1$, which would need to be mapped to the set $\setv$
appropriately. 
To this end we write $Q_{d, m} = Q_{\{1, 2, \cdots, d\}, m}$.
The formula depends on whether the quadrature rules are
nested, i.e., whether $U_i$ includes all the quadrature points from
$U_{i-1}$.

\paragraph{Non-nested case}

For non-nested one-dimensional rules, Smolyak's method can be written
explicitly as
\begin{align}
\label{eq:Sm_nonnest}
Q_{d, \mm}(g)
\,=\, \sum_{\satop{\bsi\in\bbN^d}{d\le|\bsi|\leq  d + \mm - 1}}
\sum_{k_1 = 0}^{n_{(i_1)} - 1}\cdots
\sum_{k_d = 0}^{n_{(i_d)} - 1} w_{\bsi, \bsk}\, g(\bst_{\bsi, \bsk})\,,
\end{align}
where the quadrature point $\bst_{\bsi, \bsk} \in D^d$ has coordinates
$(\bst_{\bsi, \bsk})_j = t_{i_j, k_j}$ for $j = 1, 2, \ldots, d$ and
\begin{align}
\label{eq:nonnest_weights}
w_{\bsi, \bsk} \,=\,
\sum_{\satop{\bsp \in \{0, 1\}^d}{d\le|\bsi + \bsp| \leq d + \mm - 1}}
\prod_{j=1}^d \left((-1)^{p_j}\,w_{i_j, k_j}\right) \,.
\end{align}

\paragraph{Nested case}

When the $U_i$ are nested, we assume that the quadrature points and
weights are ordered such that at level $i$ the new points occur at the end
of the point set, from index $n_{i - 1}$ onwards. That is, for all $i \in
\N$ we have $t_{i, k} = t_{i + 1, k}$ for $k = 0, 1, \ldots n_i - 1$.
Then, to ensure that the function is only evaluated at each node once,
\eqref{eq:smolyak} can be rewritten as
\begin{align}
\label{eq:Sm_nest}
Q_{d, \mm}(g)\,=\,\sum_{\satop{\bsi\in\bbN^d}{d\le|\bsi|\leq  d + \mm - 1}}
\sum_{k_1 = n_{(i_1 - 1)}}^{n_{(i_1)} - 1}\cdots
\sum_{k_d = n_{(i_d - 1)}}^{n_{(i_d)} - 1} w_{\bsi, \bsk}\, g(\bst_{\bsi, \bsk}) \, ,
\end{align}
with weights
\begin{align}
\label{eq:nest_weights}
w_{\bsi, \bsk} \,=\,
\sum_{\satop{\bsq\in\bbN^d,\,\bsq \geq \bsi}{d\le |\bsq| \leq d + \mm - 1}}
\prod_{j = 1}^d \left(w_{q_j, k_j} - w_{q_j - 1, k_j}\right)\,,
\end{align}
where we set $w_{0, k} \equiv 0$ for all $k\ge 0$ and $w_{q, k} \equiv 0$
when $k \geq n_q$. In particular, when $q_j = i_j$ in
\eqref{eq:nest_weights} the weight that is subtracted is 0, that is,
$w_{q_j - 1, k_j} = w_{i_j - 1, k_j} = 0$, since in \eqref{eq:Sm_nest}
$k_j \geq n_{i_j - 1}$.


\subsection{Smolyak quadrature via the combination technique}

The combination technique (it combines different straightforward tensor
product rules, hence the name) provides an alternative formulation to
\eqref{eq:smolyak} as follows, see, e.g., \cite{GSZ92, WW95},
\begin{align} \label{eq:comb}
  Q_{d,\mm}(g)
  &\,=\,
  \sum_{\satop{\bsi\in\bbN^{d}}{\max(\mm,d)\le |\bsi|\le d + \mm - 1}}
  (-1)^{d+\mm-1-|\bsi|}
  \binom{d-1}{|\bsi|-\mm}
  \bigg(\bigotimes_{j = 1}^d U_{i_j}\bigg)(g) \nonumber\\
    &\,=\,  \sum_{r = \max(\mm - d + 1, 1)}^{\mm}
  (-1)^{\mm - r} \binom{d -1}{d + r - 1 - \mm}
  \sum_{\satop{\bsi \in \N^{d}}{|\bsi| = d + r - 1}} \bigg( \bigotimes_{j = 1}^d U_{i_j} \bigg)(g)
  \nonumber\\
    &\,=\,  \sum_{r = \max(\mm - d + 1, 1)}^{\mm}
  (-1)^{\mm - r} \binom{d -1}{m - r}
  \sum_{\satop{\bsi \in \N^{d}}{|\bsi| = d + r - 1}} \bigg( \bigotimes_{j = 1}^d U_{i_j} \bigg)(g)
  \,,
\end{align}
where we have simplified using the symmetry of the binomial coefficient:
$\binom{n}{k} = \binom{n}{n - k}$. We adopt the usual convention that
$\binom{0}{0} := 1$.

Using \eqref{eq:comb}, we can now rewrite the MDM algorithm from the
second equality in \eqref{eq:smo1} as
\begin{align} \label{eq:smC}
  \calA^\mathrm{C}(f)
  &\,=\, c_\emptyset\,f(\bszero) +
   \sum_{\emptyset\ne\setv\in\calU_{\rm ext}}
   \sum_{\satop{\setu\in\calU}{\setu\supseteq\setv}}
  (-1)^{|\setu|-|\setv|}
    \nonumber\\
  & \quad\times\sum_{m = \max(\mm_\setu - |\setv| + 1, 1)}^{\mm_\setu}
  (-1)^{\mm_\setu - m} \binom{|\setv| -1}{m_\setu - m}
  \sum_{\satop{\bsi_\setv \in \N^{|\setv|}}{|\bsi_\setv| = |\setv| + m - 1}} \bigg( \bigotimes_{j \in \setv} U_{i_j} \bigg)(f(\cdot_\setv; \bs0)) \nonumber\\
  &\,=\, c_\emptyset\,f(\bszero) + \sum_{\emptyset\ne \setv\in\calU_{\rm ext}}
  \sum_{\satop{m=1}{\widetilde{c}(\setv,m)\ne 0}}^{\mm_{\max}}
  \widetilde{c}(\setv,m)\, \widetilde{Q}_{\setv, m}(f(\cdot_\setv; \bs0))\,,
\end{align}
where for a nonempty set $\setv$ and $m=1,\ldots,m_{\max}$ we define
\begin{align} \label{eq:smC-R}
  \widetilde{Q}_{\setv, m}(g_\setv) \,\coloneqq\,
  \sum_{\satop{\bsi_\setv \in \N^{|\setv|}}{|\bsi_\setv| = |\setv| + m - 1}} \bigg( \bigotimes_{j \in \setv} U_{i_j} \bigg)
  (g_\setv)\,,
\end{align}
and
\begin{align} \label{eq:c-comb}
  \widetilde{c}(\setv,m) \,\coloneqq\,
  \sum_{\satop{\setu\in\calU,\,\setu\supseteq\setv}{\mm_\setu - |\setv| + 1 \le m\le \mm_\setu}}
  (-1)^{|\setu|-|\setv| + \mm_\setu - m} \binom{|\setv| -1}{m_\setu - m}\,.
\end{align}

The formulation \eqref{eq:smC}--\eqref{eq:c-comb} is very similar to the
formulation \eqref{eq:smo1}--\eqref{eq:c-smol}, and the computation of the
values $\widetilde{c}(\setv,m)$ can also be done while constructing the
extended active set. This is shown in
Pseudocode~2A$^\prime$, 
which works in a similar way to Pseudocode~2A. 
The key change is that Step~12 of Pseudocode~2A 
is replaced by Steps 12--14 of Pseudocode~2A$^\prime$. 

The quantity $\widetilde{Q}_{\setv, m}(g_\setv)$ is a straightforward
tensor product quadrature rule
\begin{align} \label{eq:prod}
  \widetilde{Q}_{d, m}(g) \,\coloneqq\,
  \sum_{\satop{\bsi \in \N^d}{|\bsi| = d + m - 1}}
  \sum_{k_1=0}^{n_{(i_1)}-1} \cdots \sum_{k_d=0}^{n_{(i_d)}-1}
  w_{\bsi,\bsk}\,g(\bst_{\bsi,\bsk})\,,
\end{align}
with $w_{\bsi,\bsk} = \prod_{j=1}^d w_{i_j,k_j}$ and $(\bst_{\bsi,\bsk})_j
= t_{i_j,k_j}$ for $j=1,\ldots,d$. So the implementation of the Smolyak
MDM algorithm using the combination technique can be obtained analogously
by modifying Pseudocode~3A, 
shown in Pseudocode~3A$^\prime$. 
The essential changes are in Steps~5 and~6.
\clearpage

\begin{algorithm}[!t]
 \floatname{algorithm}{Pseudocode~2A$^\prime$} 
 \caption{(The extended active set for Smolyak with combination technique)}
 \begin{algorithmic}[1]
 \State Initialize $\Uext \leftarrow \calU$
 \Comment start from the active set
 \State Initialize $c_\emptyset \leftarrow 1$
 \Comment for $\setu=\emptyset$
 \For{$\emptyset\ne\setu \in \calU$ with $|\setu|$ from $1$ to $\supdim$}
 \Comment traverse in increasing cardinality
 \State Calculate $m_\setu$
 \Comment formula for $m_\setu$ is given from theory
 \State Update $c_\emptyset \leftarrow c_\emptyset  + (-1)^{|\setu|}$
 \State Initialize $\widetilde{c}(\setu, m) \leftarrow 0$ for $m$ from $1$ to $\mm_{\max}$
 \For{$\emptyset \ne \setv \subseteq \setu$}
 \Comment generate nonempty subsets
 \If{$\setv \not \in \Uext$}
 \Comment look up and add missing subset
 \State Add $\setv$ to $\Uext$
 \State Initialize $\widetilde{c}(\setv, m) \leftarrow 0$ for $m$ from $1$ to $\mm_{\max}$
 \EndIf
 \For{$m$ from $m_\setu - |\setv|+1$ to $m_\setu$}
 \Comment update relevant entries 
 \State Update $\widetilde{c}(\setv,m) \leftarrow \widetilde{c}(\setv,m) +
 (-1)^{|\setu|-|\setv| + \mm_\setu - m} \binom{|\setv| -1}{m_\setu - m}$
 \EndFor
 \EndFor \EndFor
\end{algorithmic}
\end{algorithm}
\begin{algorithm}[!h]
 \floatname{algorithm}{Pseudocode~3A$^\prime$} 
 \caption{(Implementing the Smolyak MDM with combination technique)}
 \begin{algorithmic}[1]
 \State Initialize $\calA^{\rm C}(f) \leftarrow c_\emptyset\times f(\bszero)$
 \For{$\emptyset\ne\setv \in \Uext$}
 \For{$m$ from $1$ to $\mm_{\max}$}
 \If{$\widetilde{c}(\setv,m)\ne 0$}
 \State Calculate $\widetilde{Q}_{\setv,m}(f(\cdot_\setv;\bszero))$ using \eqref{eq:prod}
 \State Update $\calA^{\rm C}(f) \leftarrow \calA^{\rm C}(f)
 + \widetilde{c}(\setv,m)\times \widetilde{Q}_{\setv,m}(f(\cdot_\setv;\bszero))$
 \EndIf \EndFor\EndFor \\
 \Return $\calA^{\rm C}(f)$
\end{algorithmic}
\end{algorithm}

\subsection{Direct Smolyak vs combination technique}

Here we compare the computational cost between the direct Smolyak
implementation $\calA^\mathrm{S}$ given by
\eqref{eq:smo1}--\eqref{eq:c-smol} and the combination technique
implementation $\calA^\mathrm{C}$ given by
\eqref{eq:smC}--\eqref{eq:c-comb}. Throughout, we will use the notation
$\cost(\cdot)$ to denote the whole cost, $\#(\cdot)$ to denote the number
of function evaluations, and $\$(|\setv|)$ to denote the cost of
evaluating the original integrand $f$ at some anchored point $(\bsx_\setv;
\bs0)$.

The cost of the direct Smolyak implementation is
\begin{align*}
\cost(\calA^\mathrm{S}) \,&=\,
\sum_{\emptyset\ne\setv \in \Uext} \sum_{\satop{m=1}{c(\setv,m)\ne 0}}^{m_{\max}}
\cost \Big( Q_{\setv, m}(f(\cdot_\setv; \bs0)) \Big)\,.
\end{align*}
Similarly for the combination technique the cost is
\begin{align*}
\cost(\calA^\mathrm{C}) \,&=\,
\sum_{\emptyset\ne\setv \in \Uext} \sum_{\satop{m=1}{\widetilde{c}(\setv,m)\ne 0}}^{m_{\max}}
\cost \Big( \widetilde{Q}_{\setv, m}(f(\cdot_\setv; \bs0)) \Big)\,.
\end{align*}
We expect that $\widetilde{c}(\setv,m)$ will be nonzero more often than
$c(\setv,m)$ would be nonzero, so that the combination technique
approach needs to compute more quadrature approximations. We need to
account for both the cost to construct the quadrature weights and the cost
of function evaluations.

In terms of the number of function evaluations, we have for the direct
Smolyak implementation
\begin{align}\label{eq:smolyak_nb_funevals}
\#(Q_{d, m})
\,=\, \begin{cases}
\sum_{|\bsi| \leq d + m - 1} \prod_{j=1}^d n_{i_j}
&\text{if non-nested,}\\[2mm]
\sum_{|\bsi| \leq d + m -1} \prod_{j=1}^d (n_{i_j} - n_{i_j - 1})
& \text{if nested.}
\end{cases}
\end{align}
For the combination technique, we have
\begin{align*}
 \#(\widetilde{Q}_{d, m})
 &\,=\, \textstyle \sum_{|\bsi| = d + m - 1} \prod_{j=1}^d  n_{i_j} \\
 &\,=\, \textstyle
 \sum_{|\bsi| = d + m - 1} \prod_{j=1}^d
 \left(\sum_{p_j = 1}^{i_j} n_{p_j} - n_{p_j - 1}\right) \\
 &\,=\, \textstyle  \sum_{|\bsi| = d + m - 1}
 \sum_{\bsp \in \N^{d},\,\bsp \leq \bsi} \prod_{j=1}^d  (n_{p_j} - n_{p_j - 1})\\
 &\,=\, \textstyle
 \sum_{|\bsp| \leq d + m - 1} \binom{2d+m-|\bsp|-2}{d-1} \prod_{j=1}^d (n_{p_j} - n_{p_j - 1}).
\end{align*}
So $\#(Q_{d, m}^{\rm nested}) \le \#(\widetilde{Q}_{d, m}) \le
\#(Q_{d, m}^{\rm non\mbox{-}nested})$.

For the direct Smolyak implementation we note that \eqref{eq:Sm_nonnest}
(or \eqref{eq:Sm_nest} in the nested case) using the weights
\eqref{eq:nonnest_weights} (respectively \eqref{eq:nest_weights}) is
simply a grouping of all of the quadrature points, but the total
collection of one-dimensional weights that need to be evaluated is the
same as in \eqref{eq:smolyak}; so the cost of computing the
weights is $\sum_{|\bsi| \leq d + m - 1} d \prod_{j=1}^d (n_{i_j} + n_{i_j
- 1})$. On the other hand, the cost of computing the weights in
\eqref{eq:smC-R} is clearly $\sum_{|\bsi| = |\setv| + m - 1} |\setv|
\prod_{j \in \setv} n_{i_j}$.

Thus for the direct Smolyak implementation we have
\begin{align*}
 &\cost(Q_{\setv, m}) \,=\,\\
 &\;\;
 \begin{cases}
 \sum_{|\bsi| \leq |\setv| + m - 1} \left( |\setv| \prod_{j \in \setv} (n_{i_j} + n_{i_j - 1})
 + \$(|\setv|) \prod_{j \in \setv} n_{i_j}\right) & \mbox{if non-nested}, \\
 \sum_{|\bsi| \leq |\setv| + m - 1} \left( |\setv| \prod_{j \in \setv} (n_{i_j} + n_{i_j - 1})
 + \$(|\setv|) \prod_{j \in \setv} (n_{i_j} - n_{i_j - 1})\right) & \mbox{if nested},
 \end{cases}
\end{align*}
and for the combination technique
\begin{align*}
 \textstyle\cost(\widetilde{Q}_{\setv, m})
 \,=\, \sum_{|\bsi| = |\setv| + m - 1}\left( |\setv| \prod_{j \in \setv} n_{i_j}
 + \$(|\setv|) \prod_{j \in \setv} n_{i_j}\right)\,.
\end{align*}

To summarize, the combination technique implementation is likely to
require more quadrature approximations than the direct method (more
nonzero $\widetilde{c}(\setv,m)$ than $c(\setv,m)$), and it uses more
function evaluations than the direct method in the nested case, but the
cost of computing the weights is cheaper. It is not immediately clear
which method would be the overall winner.

\section{MDM with randomized QMC} \label{sec:RQMC}

A randomly shifted version of the QMC approximation \eqref{eq:generic_qmc}
takes the form (see, e.g., \cite{DKS13})
\begin{align*}
 \frac{1}{n} \sum_{i = 0}^{n - 1} g\left(\left\{\bst^{(i)} + \bsDelta\right\}\right),
\end{align*}
where $\bsDelta\in [0,1)^d$ is the random shift with independent and
uniformly distributed components $\Delta_j \in [0,1)$ for $j=1,\ldots,d$,
and the braces around a vector indicate that we take the fractional part
of each component in the vector. The advantage of a randomly shifted QMC
method is that it provides an unbiased estimate of the integral, and
moreover, one can obtain a practical error estimate by using a number of
independent random shifts.

In this section we outline how to implement randomized QMC (RQMC) versions
of the QMC MDM from Section~\ref{sec:qmc} by random shifting. One approach
that comes to mind is to use a completely different set of independent
shifts for each $A_\setu(f(\cdot_\setv; \bs0))$ in \eqref{eq:A-qmc}. But
with this approach none of the function evaluations can be reused. Another
approach would be to use the same set of independent shifts for those
$A_\setu(f(\cdot_\setv; \bs0))$ with the same cardinality $|\setu|$, in a
similar way to how the QMC method is applied to the MDM. But this approach
also conflicts with our strategy to reuse function values based on the
position where a subset $\setv$ originates from $\setu$. Furthermore, it
is unclear in this case how to obtain a valid error estimate of the
overall MDM algorithm because of the dependence between many shifts. So,
instead of these two approaches, we will describe a third approach which
allows for a valid error estimate and the 
reuse of function values as in the case of the non-randomized QMC MDM.

Instead, our approach is to randomize the MDM algorithm itself, 
by treating it as an algorithm that acts on a function of $\tau^*$ variables,
and then independently shifting each of the $\tau^*$ variables
(recall from Lemma~\ref{lem:dim} that $\tau^*$ is the truncation dimension, i.e., the
largest index of the variables that appear in the active set). 
We generate $r$ independent random shifts 
$\bsDelta^{(1)}, \ldots,\bsDelta^{(r)} \in
[0,1)^{\tau^*}$, and we take
%
\begin{align} \label{eq:rqmc0}
  \calA^{\rm R}(f) \,=\, \frac{1}{r} \sum_{q=1}^r \calA_q(f)\,,
\end{align}
where each $\calA_q(f)$ is a shifted QMC MDM algorithm analogous to
\eqref{eq:s-qmc1}--\eqref{eq:c-qmc},
\begin{align}
\label{eq:rqmc}
 \calA_q(f) \,:=\, c_\emptyset\,f(\bszero) +
 \sum_{\emptyset\ne\setv \in \Uext}
 \sum_{\setw \in \calM(\setv)}
 \sum_{\satop{m=0}{c(\setv, \setw, m)\ne 0}}^{m_{\max}}
 c(\setv, \setw, m)\,
 \frac{\widetilde{S}_{\setv,\setw,m}^{(q)}(f)}{2^{m_{\max}}}\,,
\end{align}
but now with function values shifted by $\bsDelta^{(q)}$ in
\begin{align} \label{eq:s-rqmc}
 \widetilde{S}_{\setv,\setw,m}^{(q)}(f)
 \,:=\, \sum_{i = \lfloor 2^{m-1} \rfloor}^{2^m - 1} f\left(
 \left\{\bst^{(i)}_{\setw\to\setv} + \bsDelta^{(q)}_\setv\right\}; \bszero\right).
\end{align}

In this way, for each independent shift $\bsDelta^{(q)}$
we can perform the MDM approximation $\calA_q(f)$
using the reformulation described in Section~\ref{sec:qmc}.
Note that the reuse of function values is only possible inside the 
approximation for a single shift,
reuse across different shifts will never be possible because the
shifted approximations must be independent. 
As such, the cost of our RQMC MDM using $r$ shifts is $r$ times
the cost of the deterministic QMC MDM from Section~\ref{sec:qmc}.

\begin{algorithm}[!b]
 \floatname{algorithm}{Pseudocode~3B$'$} 
 \caption{(Implementing the RQMC MDM)}
\begin{algorithmic}[1]
 \State Initialize $\calA^{\rm R}(f) \leftarrow 0$
 \State Initialize $E \leftarrow 0$
 \For{$q$ from $1$ to $r$}
 \Comment compute MDM for each shift
 \State Generate $\bsDelta^{(q)}$ independently and uniformly from $[0,1)^{\tau*}$
 \State Initialize $\calA_q(f) \leftarrow c_\emptyset\times f(\bszero)$
 \For{$\emptyset\ne\setv \in \Uext$}
 \For{$\setw \in \calM(\setv)$}
 \For{$m$ from $0$ to $\mm_{\max}$}
 \If{$c(\setv,\setw,m)\ne 0$}
    \State Calculate $\widetilde{S}^{(q)}_{\setv,\setw,m}(f)$ using
      \eqref{eq:s-rqmc}
    \State Update $\calA_q(f) \leftarrow \calA_q(f)
       + c(\setv,\setw,m)\times \widetilde{S}^{(q)}_{\setv,\setw,m}(f) / 2^{m_{\max}}$
  \EndIf \EndFor\EndFor\EndFor
 \State Update $\calA^{\rm R}(f) \leftarrow \calA^{\rm R}(f)+ \calA_q(f)$
 \Comment sum up from different shifts
 \State Update $E \leftarrow E + \left(\calA_q(f)\right)^2$
 \Comment estimate error from different shifts
 \EndFor
 \State Compute $\calA^{\rm R}(f) \leftarrow \calA^{\rm R}(f)/r$
 \Comment final MDM approximation
 \State Compute $E \leftarrow \sqrt{[E - r \times (\calA^{\rm R}(f))^2]/[r(r-1)]}$
 \Comment final error estimate \\
 \Return $\calA^{\rm R}(f)$ and $E$
\end{algorithmic}
\end{algorithm}

Since $\calA_1(f),\ldots,\calA_r(f)$ are independent random variables each
with the same mean $\sum_{\setu\in\calU} I_\setu(f_\setu)$, their average
$\calA^{\rm R}(f)$ also has the same mean. Moreover, the variance of
$\calA^{\rm R}(f)$ is $1/r$ times the variance of each $\calA_q(f)$. The
quadrature component of the root-mean-square error, i.e., $\sqrt{\bbE\,
|\sum_{\setu\in\calU} I_\setu(f_\setu) - \calA^{\rm R}(f)|^2}$, can be
estimated from these $r$ sample values by
\begin{equation*}
 \sqrt{\frac{1}{r(r-1)} \sum_{q=1}^r \left(\calA^{\rm R}(f) - \calA_q(f)\right)^2}
 \,=\, \sqrt{\frac{1}{r(r-1)} \left(\sum_{q=1}^r \left(\calA_q(f)\right)^2 - r\left(\calA^{\rm R}(f)\right)^2\right)}.
\end{equation*}

The computation of $\calM(\setv)$ and $c(\setv, \setw, m)$ is exactly the
same as in Pseudocode~2B. We only need to replace Pseudocode~3B by
Pseudocode~3B$'$, where we include also the error estimation.

Note that due to linearity in \eqref{eq:rqmc0}--\eqref{eq:rqmc}, we can
also interpret $\calA^\mathrm{R}(f)$ as one MDM algorithm where each of
the quadrature approximations $\calA_\setu$ is obtained by the average of
$r$ randomly shifted QMC rules. In this case we can also estimate the
root-mean-square error corresponding to each~$\setu$. However, note that
the shifts between different sets~$\setu$ are correlated and an
estimate for the total variance cannot be obtained by simply summing up
the individual variance estimates.

\section{Computing the tolerance $T$} \label{sec:T}

Under the theoretical setting of the paper~\cite{KNPSW17}, the
threshold parameter $T$ to construct an active set \eqref{eq:active} takes
the form (see Formula~(26) of \cite{KNPSW17})
\begin{align}
\label{eq:T}
 T \,=\,
\left(\frac{\e/2}{\sum_{\setu\subset\bbN} [w(\setu)]^{1/\alpha}}\right)^
{\alpha/(\alpha - 1)}\,,
\end{align}
where $\e>0$ is a given error tolerance parameter, while $\alpha>1$ is a
free parameter with the constraint that $\sum_{\setu\subset\bbN}
[w(\setu)]^{1/\alpha} < \infty$.
Since the active set constructed by Pseudocode~1 is uniquely defined
by $T$ and $w(\setu)$, computing the tolerance is another important step in
implementing the MDM. In this section we outline how to estimate $T$ by 
presenting novel estimates on the sum in \eqref{eq:T}.

In a practical implementation, we need to estimate the infinite sum in the
denominator \emph{from above}, to yield an underestimate of $T$, so that
the required theoretical error tolerance $\e$ is guaranteed, at the
expense of enlarging the active set. Understandably, we should aim for a
tight estimate of the infinite sum to avoid making the active set too
large and thus the algorithm too expensive. Similarly, the free parameter
$\alpha$ should be chosen so as to make the threshold parameter $T$ as
large as possible.

Here we derive upper bounds on the infinite sum by assuming further
structure in $w(\setu)$, namely, that the POD form of $w(\setu)$, see
\eqref{eq:POD}, is further parametrized for $\ell\ge 0$ and $j\ge 1$ by
\begin{equation}\label{eq:omegas_sum}
\Omega_{\ell} \,=\, c_1(\ell!)^{b_1}
\quad \text{and} \quad
\omega_j = c_2j^{-b_2},
\end{equation}
with $b_1\ge 0$, $b_2>1$, $b_2>b_1$, and $c_1,c_2>0$. Thus
\begin{align} \label{eq:headache}
  \sum_{\setu\subset\bbN} [w(\setu)]^{1/\alpha}
  &\,=\, \sum_{\ell=0}^\infty \sum_{|\setu|=\ell}
  \bigg(c_1 (\ell!)^{b_1}
  \prod_{j\in\setu} (c_2 j^{-b_2}) \bigg)^{1/\alpha} 
  \,=\, c_1^{1/\alpha}
  \sum_{\ell=0}^\infty (\ell!)^a\, c^\ell
  \sum_{|\setu|=\ell} \prod_{j\in\setu} j^{-b},
\end{align}
with
\begin{equation}\label{eq:symbols}
  a\,=\,\frac{b_1}\alpha,\quad b\,=\,\frac{b_2}\alpha,\quad
  \mbox{and}\quad c\,=\, c_2^{1/\alpha}.
\end{equation}

We know from \cite[Lemma~10]{KNPSW17} that the infinite sum in
\eqref{eq:headache} is finite if $b>1$ and $b>a$. In turn this means that
the free parameter $\alpha>1$ in \eqref{eq:T} should satisfy $\alpha <
b_2$.

In the next two lemmas we obtain an upper bound on the infinite sum in
\eqref{eq:headache} for somewhat general parameters $a,b,c$, without
taking into account the constraints in how they relate to each other or
how they relate to $\alpha$. After the two lemmas we will discuss when and
how the lemmas can be applied in our case.

\begin{lemma} \label{lem:sum-prod}
For any $b>1$ and $\ell\ge 1$ we have
\begin{equation}\label{eq:finito}
  \sum_{|\setu|=\ell}\prod_{j\in\setu}j^{-b}
  \,\le\,
  \frac{z^{\ell-1}}{(\ell-1)!} \left(1+\frac{z}{\ell}\right),
  \qquad
  z\,\coloneqq\,\frac{(2/3)^{b-1}}{b-1}.
\end{equation}
\end{lemma}

\begin{proof}
By identifying $\setu$ with the vector $(j_1,j_2,\ldots,j_\ell)$ with
ordered elements $j_1<j_2<\cdots< j_\ell$, we see that
\begin{equation}\label{eq:iniziare}
  \sum_{|\setu|=\ell}\prod_{j\in\setu}j^{-b}
  \,=\,
  \sum_{j_1=1}^\infty j_1^{-b}
  \sum_{j_2= j_1+1}^\infty j_2^{-b} \cdots
  \sum_{j_{\ell-1}= j_{\ell-2}+1}^\infty j_{\ell-1}^{-b}
  \sum_{j_\ell= j_{\ell-1}+1}^\infty j_\ell^{-b}.
\end{equation}

For any $k\in\bbN$ and $p>1$ we have
\[
  \sum_{j=k+1}^\infty j^{-p}
 \le \int_{k+1/2}^\infty x^{-p}\,\rd x  \,=\, \frac{(k+1/2)^{-(p-1)}}{p-1}
 \,<\, \frac{k^{-(p-1)}}{p-1},
\]
which holds since the mid-point rule underestimates this integral.
Therefore, the very last sum in \eqref{eq:iniziare} is bounded by
\[
   \int_{j_{\ell-1}+1/2}^\infty x^{-b}\,\rd x
   \,=\, \frac{(j_{\ell-1}+1/2)^{-(b-1)}}{b-1}
   \,<\, \frac{j_{\ell-1}^{-(b-1)}}{b-1},
\]
and in turn the last two sums in \eqref{eq:iniziare} are bounded by
\[
  \frac1{b-1}\int_{j_{\ell-2}+1/2}^\infty x^{-(2b-1)} \,\rd x
  \,=\, \frac{(j_{\ell-2}+1/2)^{-2(b-1)}}{2(b-1)^2}
  \,<\, \frac{j_{\ell-2}^{-2(b-1)}}{2(b-1)^2}.
\]
Similarly, the last $\ell-1$ sums are bounded by
\[
  \frac{(j_1+1/2)^{-(\ell-1)(b-1)}}{(\ell-1)!\,(b-1)^{\ell-1}}.
\]
The sum with respect to $j_1$ has to be treated differently since we do
not want to integrate over $[1/2,\infty)$. For that purpose, we treat
differently $j_1=1$ and $j_1\ge2$, to obtain
\begin{align*}
  \sum_{|\setu|=\ell}\prod_{j\in\setu}j^{-b}
  &\,\le\, \frac1{(\ell-1)!\, (b-1)^{\ell-1}}
      \sum_{j_1=1}^\infty j_1^{-b} (j_1+1/2)^{-(\ell-1)(b-1)}\\
  &\,<\, \frac{(2/3)^{(\ell-1)(b-1)}} {(\ell-1)!(b-1)^{\ell-1}}
   + \frac1{(\ell-1)!\,(b-1)^{\ell-1}}
      \sum_{j_1=2}^\infty j_1^{-\ell\, b +(\ell-1)}.
\end{align*}
Applying the integration estimate again to the sum over
$j_1\ge2$, we get that
\[
  \frac1{(\ell-1)! (b-1)^{\ell-1}}
      \sum_{j_1=2}^\infty j_1^{-\ell\,b +(\ell-1)}\le
   \frac1{\ell!(b-1)^\ell} \left(\frac23\right)^{\ell(b-1)}.
\]
Combining the last two steps yields the estimate in \eqref{eq:finito}.
\end{proof}

\begin{lemma} \label{lem:sum}
For every $a\in(0,1)$, $b>1$, $c>0$, $s\in\bbN$ and $t\in(0,1)$, we have
\[
 \sum_{\ell=0}^\infty (\ell!)^a\,c^\ell\,
 \sum_{|\setu|=\ell}\prod_{j\in\setu}j^{-b}
 \,\le\, 1+\sum_{\ell=1}^s (\ell!)^a\,c^\ell\,
 \frac{z^{\ell-1}}{(\ell-1)!}\left(1+\frac{z}{\ell}\right) + E_{s,t},
\]
with $z \coloneqq (2/3)^{b-1}/(b-1)$ as in \eqref{eq:finito}, and
\begin{align} \label{eq:Edt}
  E_{s,t} &\,\coloneqq\, c \bigg(1+\frac{z}{s+1}\bigg)\,
  \bigg[
 \frac{t^{s/a}}{1-t^{1/a}}\bigg(s+\frac{1}{1-t^{1/a}}\bigg)\bigg]^a \nonumber\\
 &\qquad\qquad\times
 \bigg[\exp\bigg(\left(\frac{c\,z}t\right)^{1/(1-a)}\bigg)\,
 \min\left(1, \left(\frac{c\,z}t\right)^{s/(1-a)}\right)\,\frac1{s!}\bigg]^{1-a}.
\end{align}
\end{lemma}

\begin{proof}
The result is obtained by applying Lemma~\ref{lem:sum-prod} and then
estimating the tail sum from $\ell = s+1$ by $E_{s,t}$. Indeed, we have
\begin{align*}
 &\sum_{\ell=s+1}^\infty(\ell!)^a\,c^\ell\,
 \frac{z^{\ell-1}}{(\ell-1)!}\left(1+\frac{z}{\ell}\right) \\
 &\,\le\, c\left(1+\frac{z}{s+1}\right)\,
 \sum_{\ell=s+1}^\infty \ell^a\,t^{\ell-1}\,\left(\frac{c\,z}{t}\right)^{\ell-1}
 \,\frac1{[(\ell-1)!]^{1-a}}\\
 &\,\le\, c \left(1+\frac{z}{s+1}\right)\,
 \bigg[\sum_{\ell=s+1}^\infty \ell\,t^{(\ell-1)/a}\bigg]^a\,
 \bigg[\sum_{\ell=s+1}^\infty\frac1{(\ell-1)!}\,
 \left(\frac{c\,z}t\right)^{(\ell-1)/(1-a)}\bigg]^{1-a},
\end{align*}
where the last step follows from H\"older's inequality with the conjugate
pair $1/a$ and $1/(1-a)$.

Consider the equality
\[
 \sum_{\ell=s+1}^\infty \ell\,x^{\ell-1}
 \,=\, \frac{\rd}{\rd x} \left(\frac{x^{s+1}}{1-x}\right)
 \,=\, \frac{x^s}{1-x}\,\left(s+\frac{1}{1-x}\right)\,,
\]
which after substituting in $x=t^{1/a}$ yields the third factor in
\eqref{eq:Edt}.

We also have
\begin{align*}
 \sum_{\ell=s+1}^\infty\frac{y^{(\ell-1)/(1-a)}}{(\ell-1)!}
 &\,=\,\frac{y^{s/(1-a)}}{s!}
 \sum_{\ell=0}^{s-1}\frac{y^{\ell/(1-a)}}{\ell!}
 \,\le\, \exp(y^{1/(1-a)})\,\min\left(1, \frac{y^{s/(1-a)}}{s!}\right),
\end{align*}
with last inequality due to Taylor's theorem. Substituting $y=c\,z/t$
yields the fourth factor in \eqref{eq:Edt}.
\end{proof}

The following corollary summarizes the upper bound on the sum. It
introduces two new free parameters: $s \in \N$ and $t\in (0,1)$.

\begin{corollary}\label{cor:sum}
Let $w(\setu)$ have product and order dependent components satisfying
\eqref{eq:omegas_sum} with $b_2 > 0$, also let $\alpha \geq 1$ and $\alpha
\in (b_1, b_2)$. Then for every $s \in \N$ and every $t\in (0,1)$ we have
\begin{align} \label{eq:sumbound}
 \sum_{\setu\subset\bbN} [w(\setu)]^{1/\alpha}
 \,\le\, &c_1^{1/\alpha}
 \left(1+\sum_{\ell=1}^s(\ell!)^a\,c^\ell\,
 \frac{z^{\ell-1}}{(\ell-1)!}\,\left(1+\frac{z}\ell\right)
 + E_{s,t} \right),
\end{align}
where $a,b,c$ are specified in \eqref{eq:symbols}, and $z$ and $E_{s,t}$ are
as defined in Lemma~\ref{lem:sum}.
\end{corollary}

Note that when $\alpha = 1$ the upper bound \eqref{eq:sumbound} on the sum
is valid, even though in this case the tolerance $T$ in \eqref{eq:T} is
undefined. 

\begin{remark} If $b_1 = 0$ so that $w(\setu)$ is of product form,
then $a = 0$ in \eqref{eq:symbols} and the sum can be bounded above using
the same method as in \cite[equation (7)]{GilW17}, namely, for every $s \in \N$,
\begin{align*}
\sum_{\setu \subset \N} [w(\setu)]^{1/\alpha}
\,&\leq\, c_1^{1/\alpha}
\exp\left(\frac{c}{(b - 1)(s + \frac{1}{2})^{b - 1}}\right)
\prod_{j = 1}^s \left(1 + cj^{-b}\right)\,,
\end{align*}
where $b, c$ are as given in \eqref{eq:symbols}. In this case the size of
the active set will be smaller than the general case because with $b_1 =
0$ each $w(\setu)$ is smaller so the exact value of the sum $\sum_{\setu
\subset \N} [w(\setu)]^{1/ \alpha}$ is smaller, and moreover our upper bound
on the sum is tighter.

Note that the threshold parameter in the construction of optimal and
quasi-optimal active sets in \cite{GilW17} (see Remark~\ref{rem1}) also
requires a good upper bound on the sum with $\alpha = 1$.
\end{remark}

\section{Numerical experiments} \label{sec:num}

Until now we have ignored any theoretical details of the MDM and
focussed purely on the implementation given the arbitrary input parameters
$\e$, $\{\omega(\setu)\}_{|\setu| < \infty}$, $T$ and $\{m_\setu\}_{\setu
\in \calU}$. Below we give some brief details on how to specify these
parameter values for a particular test integrand following the setting of
\cite{KNPSW17}. We compare between QMC MDM and Smolyak MDM, as well as
demonstrate the speedup of our efficient implementations over the naive
implementations.

\subsection{Test integrand}

We consider the integrand $f : [-\tfrac{1}{2}, \tfrac{1}{2}]^\N
\rightarrow \R$ given by
\begin{align}
 \label{eq:f_eg}
 f(\bsx) \,=\, \frac{1}{1 + \sum_{j \geq 1} x_j/j^\beta} \,,
\end{align}
with different parameters $\beta >1$. This integrand with $\beta=2$ has
previously been considered in, e.g., \cite[Example~5]{KNPSW17}.
Although we do not know the exact value of the integral, we are able
to calculate a good approximation as a reference value and use this
reference value to calculate the total error of our MDM algorithms. This
integrand is considered a prototype function for some PDE
applications, see e.g., \cite{KSS12}.

In the theoretical setting of \cite{KNPSW17}, it is assumed that each
decomposition function $f_\setu$ belongs to some normed space $F_\setu$,
and that bounds on the norms are known, i.e., $\|f_\setu\|_{F_\setu} \le
B_\setu$ for all $\setu\ne\emptyset$.

In the particular case of the
anchored decomposition combined with the anchored norm over the shifted
unit cube $[-\tfrac{1}{2},\tfrac{1}{2}]^{|\setu|}$, we have
\begin{align} \label{eq:norm}
\nonumber
 \nrm{f_\setu}{F_\setu}
 &=
 \Bigg(\int_{[-\tfrac{1}{2},\tfrac{1}{2}]^{|\setu|}} \left(\pd{|\setu|}{}{\bsx_\setu} f_\setu(\bsx_\setu)\right)^2
 \,\rd \bsx_\setu\Bigg)^{1/2}\!\\
 &=
 \Bigg(\int_{[-\tfrac{1}{2},\tfrac{1}{2}]^{|\setu|}} \left(\pd{|\setu|}{}{\bsx_\setu} f(\bsx_\setu; \bs0)\right)^2
 \,\rd \bsx_\setu\Bigg)^{1/2},
\end{align}
and the induced operator norm of $I_\setu$ in $F_\setu$
is given by $\|I_\setu\| = 12^{-|\setu|/2} =:C_\setu$
(see \cite[Subsection~5.3]{KNPSW17}). Moreover, for the integrand
\eqref{eq:f_eg} this norm is bounded by (see
\cite[Subsection~5.4]{KNPSW17} for derivation of the case $\beta = 2$)
\[
\nrm{f_\setu}{F_\setu}
\,\leq\,
\left(1 - \tfrac{1}{2}\zeta(\beta)\right)^{-(|\setu| + 1)}
|\setu|!\prod_{j \in \setu} j^{-\beta}
\,=:\, B_\setu,
\]
where $\zeta(x) = \sum_{k = 1}^\infty k^{-x}$ for $x>1$ is the Riemann
zeta function.

\subsection{Active set construction}
We continue under the assumption that $f_\setu$ is in the Hilbert
space whose norm is given by \eqref{eq:norm}.

Following \cite[Formula~(26)]{KNPSW17}, the above description leads to an
active set \eqref{eq:active} with $w(\setu) := C_\setu B_\setu$, which is
in POD form \eqref{eq:omegas_sum} with
\begin{align*}
c_1 \,=\, \frac{1}{1 - \frac{1}{2}\zeta(\beta)}\,,
\quad
c_2 \,=\, \frac{c_1}{\sqrt{12}}\,,
\quad
b_1\,=\, 1,\,
\quad
b_2 \,=\, \beta\,,
\end{align*}
and where the tolerance $T$ is given by \eqref{eq:T}.
We estimate $T$ by the upper bound \eqref{eq:sumbound},
which we compute by fixing $s=1000$,
$t=0.5$, and then maximizing the estimated tolerance for $\alpha$ over $100$ equispaced
points in $(1,\beta)$.

Table~\ref{tab:active} presents details on the active set for the
parameter values $\beta = 4,3,2.5$ and $\e = 10^{-1}, 10^{-2}, 10^{-3}$.
The rows are labelled as follows: $\e$ is the error request, $T$ is the
computed tolerance for the active set, $\supdim$ is the superposition dimension,
$\tau^*$ is the maximum truncation dimension, and the last ten rows
display the number of sets of each size in the active set.

We see that although there are many sets to consider in MDM, even in
the hardest case with $\beta = 2.5$ and $\e=10^{-2}$ we only ever deal
with integrals up to $10$ dimensions, with the highest coordinate
considered being $24724$.

Below we will restrict ourselves to the case $\beta=3$, with error request
down to $\e=10^{-6}$.

Actually, the above approach from \cite{KNPSW17} for prescribing the
parameters $w(\setu)$ and $T$ overestimates the truncation error for our
test integrand and makes the active set much larger than necessary. A
tighter truncation error estimate can be done for this test integrand
using a Taylor series argument (see e.g., \cite[Remark~13]{KNPSW17}) and
this should yield better input parameters $w(\setu)$ and $T$ to our
efficient MDM algorithms. Analysis on the best strategy to prescribe the
active set parameters for any given practical integrand falls outside the
scope of this paper.

\begin{table}[!h]
\centering \small
\begin{tabular}{r|r@{ }r@{ }r|r@{ }r@{ }r|r@{ }r}
\hline
\hline
&&& $\beta=4$ &&& $\beta=3$ && $\beta =2.5$ \\
\hline
\hline
 $\e$ & 1e-1 & 1e-2 & 1e-3 & 1e-1 & 1e-2 & 1e-3 & 1e-1 & 1e-2 \\
 $T$ & 1.4e-4 & 2.8e-6 & 6.4e-8 & 4.0e-6 & 3.6e-8 & 3.8e-10 & 1.5e-8 & 4.9e-11 \\
 $\supdim$ & 3 & 4 & 5 & 5 & 6 & 7  & 8 & 10 \\
 $\tau^*$ & 10 & 28 & 72  & 86 & 418 & 1907 & 2528 & 24724 \\
 \hline
 size 1 &  9 & 26 &  68 &  76 &  370 &  1686 &  2019 &  19750 \\
      2 & 12 & 48 & 159 & 195 & 1285 &  7327 & 10077 & 126882 \\
      3 &  5 & 28 & 132 & 202 & 1828 & 13117 & 21996 & 354377 \\
      4 &  0 &  4 &  36 &  80 & 1234 & 11907 & 26258 & 559155 \\
      5 &  0 &  0 &   1 &  10 &  361 &  5578 & 17874 & 536133 \\
      6 &  0 &  0 &   0 &   0 &   32 &  1145 &  6513 & 313623 \\
      7 &  0 &  0 &   0 &   0 &    0 &    69 &  1088 & 106877 \\
      8 &  0 &  0 &   0 &   0 &    0 &     0 &    47 &  18582 \\
      9 &  0 &  0 &   0 &   0 &    0 &     0 &     0 &   1210 \\
     10 &  0 &  0 &   0 &   0 &    0 &     0 &     0 &      8 \\
\hline
 \hline
\end{tabular}
\caption{\small Results from the active set construction for various
$\beta$ and $\e$.} \label{tab:active}
\end{table}

\subsection{QMC MDM}

For the randomized QMC MDM we use an extensible rank-1 lattice rule with
generating vector
\begin{align*}
  \bsz &=
  ( 1,  756581,  694385,  178383,  437131,  945527,   62405, 1079809,  \\
  &
  \qquad  991997,  750785,  187845, 1666795,  491701, 1092667, \\
  &
  \qquad    1279469,  817683, 1946073, 1946073, 1530387,  686611, \ldots )
  ,
\end{align*}
with $n=2^{m}$ points for any $m=0,\ldots,25$. It is constructed by a
component-by-component (CBC) algorithm as outlined in \cite{CKN06},
but the search criterion was appropriately modified to match the norm
\eqref{eq:norm}. Also, we apply the tent-transform (see \cite{Hic02}) $x
\mapsto 1  - |2x - 1|$ to each component of the shifted quadrature points
in $[0,1]$, and then translate it to $[-\frac{1}{2}, \frac{1}{2}]$.

Following \cite{KNPSW17}, we assume that the quadrature error for each
decomposition term $f_\setu$ to be bounded by $G_{\setu,q} \,
(n_{\setu}+1)^{-q} \, \|f_\setu\|_{F_\setu}$ with appropriate constants
$G_{\setu,q}$ and~$q$; a Lagrange optimization argument then results in
choosing the number of points $n_\setu \ge h_\setu$, with
\begin{equation}\label{eq:h_u}
   h_\setu =
  \bigg(\frac{2}{\e}\sum_{\setv\in \setU}
  \pounds(|\setv|)^{q/(q+1)}\,(G_{\setv,q}\,B_{\setv})^{1/(q+1)}
     \bigg)^{1/q}
  \left(\frac{G_{\setu,q}\,B_{\setu}}
      {\pounds(|\setu|)}\right)^{1/(q+1)},
\end{equation}
where $\pounds(|\setu|)$ is the cost of evaluating the decomposition term
$f_\setu$.

Although the above formula for $h_\setu$ is precisely
\cite[Formula~(28)]{KNPSW17}, how we specify the other parameters here
will deviate from \cite{KNPSW17}. Based on~\eqref{eq:explicit} we set
$\pounds(|\setu|) = \max(2^{|\setu|} |\setu|,1)$. The constant
$G_{\setu,q}$ arising from the theoretical error bound is far too large,
so we set instead $G_{\setu,q}=1$ after some experiments (see below). In
the theoretical setting with the norm \eqref{eq:norm} involving mixed
first order derivatives, we expect only up to first order convergence,
leading to the choice $q=1$. However, it is known from \cite{Hic02} that
randomly-shifted and then tent-transformed lattice rules can achieve
nearly second order convergence if the integrand has mixed second order
derivatives; thus we take instead $q=2$ (see justification by experiments
below) without formally switching the setting. Finally we set $n_\setu =
2^{m_\setu}$ with $m_\setu = \max(\lceil\log_2(h_\setu)\rceil,0)$.

\begin{figure}[!b]
\centering
\includegraphics[scale=0.85]{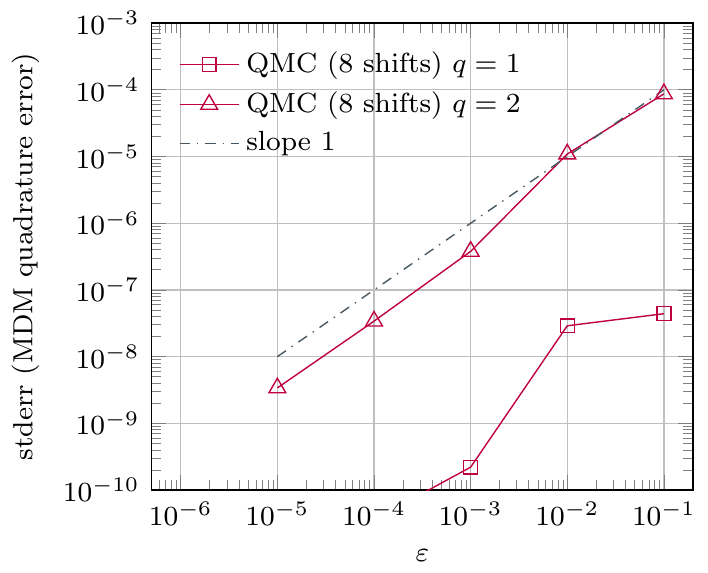} \quad
\includegraphics[scale=0.85]{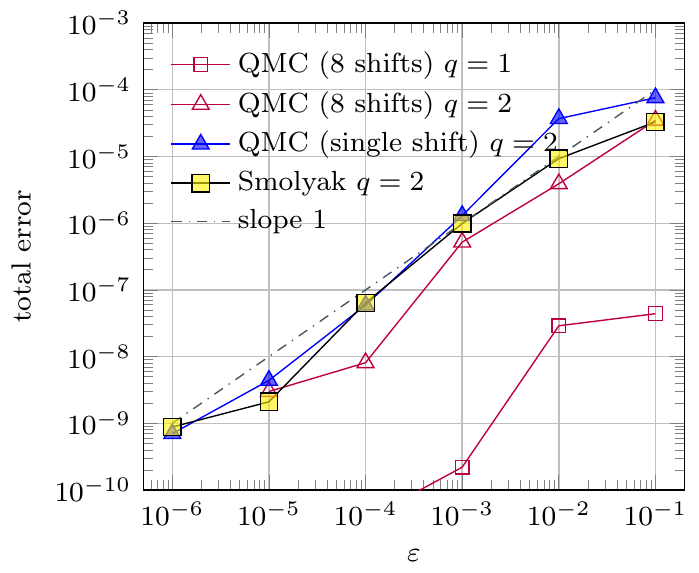}
\caption{\small Error request $\e$ against estimated standard error (left) and total error (right).} \label{fig:q2}
\vspace{-6mm}
\end{figure}

In Figure~\ref{fig:q2} on the left we plot the error request $\e$ against
the estimated standard error obtained using $16$ random shifts. This gives
an idea of the quadrature error alone. By taking $q=2$ instead of $q=1$,
we see that the error request $\e$ and the estimated standard error are in
agreement up to a constant factor. It is possible to tune this constant
factor by changing the value of $G_{\setu,q}$ (it should be at least as
large as $\|I_\setu\| = 12^{-|\setu|/2}$ and indeed the theoretical bound
yields $G_{\setu,q} = g^{|\setu|}$ with some $g>1$), but taking
$G_{\setu,q}=1$ appears to give reasonable results in practice (if it is
too big then the quadrature rules will do too much work, while if it is
too small they will underperform).

In Figure~\ref{fig:q2} on the right we plot the error request $\e$ against
the estimated total error (combining both the truncation error and the
quadrature error) by comparing the results with a reference solution
obtained using a higher precision QMC quadrature rule. The reference
solution was computed using $2^{22}$ QMC points and 16~shifts in quad
precision in 600~dimensions and resulted in a standard error of~$8\times 10^{-13}$. A
similar computation in 800~dimensions with $2^{20}$ QMC points agreed up
to 10~digits. The two graphs in Figure~\ref{fig:q2} show the same trend,
indicating that the truncation error is no worse than the quadrature
error.

\subsection{Smolyak MDM}

For the Smolyak MDM we will use the composite trapezoidal rule as the
one-dimensional rules. These rules are nested. We set $U_0 \coloneqq 0$ to
be the zero algorithm, and following \cite[Section~3]{GG98} we take the
first one-dimensional quadrature rule $U_1$ to be the single (mid-)point
rule, i.e., $n_1 \coloneqq 1$ with point $t_{1, 0} \coloneqq 0$ and weight
$w_{1, 0} \coloneqq 1$. This extra level ensures that the number of points
does not grow too quickly. Then for each $i\ge 2$ we take the
one-dimensional quadrature rule $U_i$ to be the composite trapezoidal rule
in $[-\frac{1}{2}, \frac{1}{2}]$ with $n_i := 2^{i - 1} + 1$ points at
multiples of $1/2^{i-1}$, with the weights being $1/2^{i-1}$ for the
interior points and $1/2^i$ at the two end points $\pm 1/2$.

To ensure that we iterate the points in a nested fashion, we label the
points in the order of $0$, $\pm 1/2$, $\pm 1/4$, $\pm 1/8$, $\pm
3/8,\ldots$ and so on. Explicitly, for $i\ge2$ we can write
\[
 t_{i, k} \,:=\,\begin{cases}
 k/2^p - 1/2 & \text{if } k \text{ is odd, and $p$ is such that $2^{p-1} < k < 2^p$},\\
 - t_{i, k - 1} & \text{if } k \text{ is even,}
 \end{cases}
\]
with the corresponding weights
\begin{align*}
w_{i, k} \,\coloneqq\, \begin{cases}
1/2^{i} & \text{if } k = 1, 2,\\
1/2^{i-1} & \text{if } k = 0, 3, 4, \dots, n_i - 1\,.
\end{cases}
\end{align*}

We choose the approximation levels $m_\setu$ of our quadrature rules in
the direct Smolyak MDM implementation to be such that the number of
function evaluations, see \eqref{eq:smolyak_nb_funevals} with nested
points, is at least $h_\setu$ given by the formula \eqref{eq:h_u}, with
the same definitions of $\pounds(|\setu|) = \max(2^{|\setu|} |\setu|,1)$,
$G_{\setu,q}=1$ and $q=2$. The justification for taking $q=2$ here is that
composite trapezoidal rules are known to give second order convergence for
sufficiently smooth integrands in one dimension, and this convergence is
transferred, modulo log-factors, to the Smolyak rule for multivariate
integrands with sufficient smoothness, see e.g., \cite{GG98}. We use the
same values of $m_\setu$ for the combination technique variant even though
the actual numbers of function evaluations are higher.

In particular, our chosen values of $m_\setu$ mean that the naive
implementations of Smolyak MDM and QMC MDM would use roughly the same
number of function evaluations for each $f_\setu$. However, the actual
number of function evaluations for our efficient implementations would be
lower and hence achieve the savings we aim for; see the timings in the
next subsection.

In Figure~\ref{fig:q2} on the right we also plot the error request $\e$
against the estimated total error of the direct Smolyak MDM implementation
compared with the same reference solution as for QMC MDM. We see that all
results are comparable.

\subsection{Timing results}\label{sec:quad_res}

In Table~\ref{tab:time} we present results on the run-time of our
efficient MDM implementations compared with the naive implementations for
error request $\e$ from $10^{-1}$ down to $10^{-6}$. We include results
for QMC MDM with a single random shift, direct Smolyak MDM, and Smolyak
MDM based on combination technique (CT-Smolyak). We report the total error
with respect to the same reference solution in each case, as well as the run-time
in seconds. The QMC results can vary between runs depending on the random
shift. All calculations were done in x86 long double precision on a single
node of the UNSW Katana cluster with an Intel Xeon X5675 3.07GHz CPU.

The values of $t_\mathrm{act}$ in the first column are the time in seconds used to
construct the active set (Pseudocode 1), while the values of $t_\mathrm{ext}$ are the
average time used to construct the extended active set and compute the
three sets of coefficients \eqref{eq:c-smol}, \eqref{eq:c-comb},
\eqref{eq:c-qmc} (Pseudocode 2A, 2A$'$, 2B).

We clearly see increasing speedup of the efficient implementations
over the naive ones as $\e$ decreases. The reformulation of the MDM
algorithm into this efficient formulation is the main result of this paper.

We see more speedup in the case of Smolyak MDM compared with QMC MDM. This
is as expected, because for QMC MDM there is extra work in managing the
different positions that a nonempty set can originate from as a subset of
another set in the active set: we need to cope with a more complicated
data structure for \eqref{eq:c-qmc} compared with \eqref{eq:c-smol} or
\eqref{eq:c-comb}, and we need more function evaluations. Additionally, we
expect the QMC algorithms to be much more efficient when the truncation
dimension goes up (i.e., when $\e$ goes down), since the sizes of the
Smolyak grids based on trapezoidal rules then increase faster than the 
powers of $2$ of the QMC algorithms.

\begin{table}[!h]
 \centering \footnotesize
 \begin{tabular}{l|l|ll|ll|l@{}}
\multicolumn{7}{l}{$\beta=3$, reference value $=1.1011984577041$} \\
\hline\hline
&& \multicolumn{2}{c|}{efficient} & \multicolumn{2}{c|}{naive} & \\
$\varepsilon$ & method & total error & time (s) & total error & time (s) & speedup \\
\hline\hline
1e-01                   & 
QMC        & 7.57e-05   & 0.0017576  & 7.57e-05   & 0.0032837  & 1.9 \\
\hspace*{1mm} $t_\mathrm{act}=$ 0.000768 & Smolyak    & 3.26e-05   & 0.0047466  & 3.26e-05   & 0.0063622  & 1.3 \\
\hspace*{1mm} $t_\mathrm{ext}=$ 0.00339 & CT-Smolyak & 3.26e-05   & 0.0042816  & 3.26e-05   & 0.0088774  & 2.1 \\
\hline
1e-02                   & 
QMC        & 3.66e-05   & 0.062643   & 3.66e-05   & 0.067456   & 1.1 \\
\hspace*{1mm} $t_\mathrm{act}=$ 0.00899 & Smolyak    & 9.34e-06   & 0.074826   & 9.34e-06   & 0.12321    & 1.6 \\
\hspace*{1mm} $t_\mathrm{ext}=$ 0.049 & CT-Smolyak & 9.34e-06   & 0.073692   & 9.34e-06   & 0.19568    & 2.7 \\
\hline
1e-03                   & 
QMC        & 1.26e-06   & 0.4301     & 1.26e-06   & 1.1336     & 2.6 \\
\hspace*{1mm} $t_\mathrm{act}=$ 0.0401 & Smolyak    & 9.92e-07   & 0.49712    & 9.92e-07   & 1.9859     & 4.0 \\
\hspace*{1mm} $t_\mathrm{ext}=$ 0.339 & CT-Smolyak & 9.92e-07   & 0.48502    & 9.92e-07   & 3.4984     & 7.2 \\
\hline
1e-04                   & 
QMC        & 5.90e-08   & 4.8547     & 5.90e-08   & 15.766     & 3.2 \\
\hspace*{1mm} $t_\mathrm{act}=$ 0.34 & Smolyak    & 6.39e-08   & 5.5186     & 6.39e-08   & 27.89      & 5.1 \\
\hspace*{1mm} $t_\mathrm{ext}=$ 4.08 & CT-Smolyak & 6.39e-08   & 5.5692     & 6.39e-08   & 53.191     & 9.6 \\
\hline
1e-05                   & 
QMC        & 4.41e-09   & 47.64      & 4.51e-09   & 188.61     & 4.0 \\
\hspace*{1mm} $t_\mathrm{act}=$ 2.79 & Smolyak    & 2.13e-09   & 54.331     & 2.12e-09   & 346.79     & 6.4 \\
\hspace*{1mm} $t_\mathrm{ext}=$ 41.7 & CT-Smolyak & 2.11e-09   & 56.083     & 2.12e-09   & 734.87     & 13.1 \\
\hline
1e-06                   & 
QMC        & 7.01e-10   & 442.8      & 4.31e-10   & 2163.1     & 4.9 \\
\hspace*{1mm} $t_\mathrm{act}=$ 20.2 & Smolyak    & 8.76e-10   & 504.78     & 1.14e-09   & 4093.9     & 8.1 \\
\hspace*{1mm} $t_\mathrm{ext}=$ 435 & CT-Smolyak & 2.08e-09   & 535.74     & 1.14e-09   & 9255.4     & 17.3 \\
\hline
\hline
\end{tabular}
 \caption{\small Timing comparisons between efficient and naive MDM
 implementations. 
 $t_\mathrm{act}$ is the time to construct $\calU$ and $t_\mathrm{ext}$ 
 is the average time to construct $\Uext$ for the 3 reformulations.} \label{tab:time}
\end{table}

\begin{figure}[!h]
\centering
\includegraphics[scale=0.85]{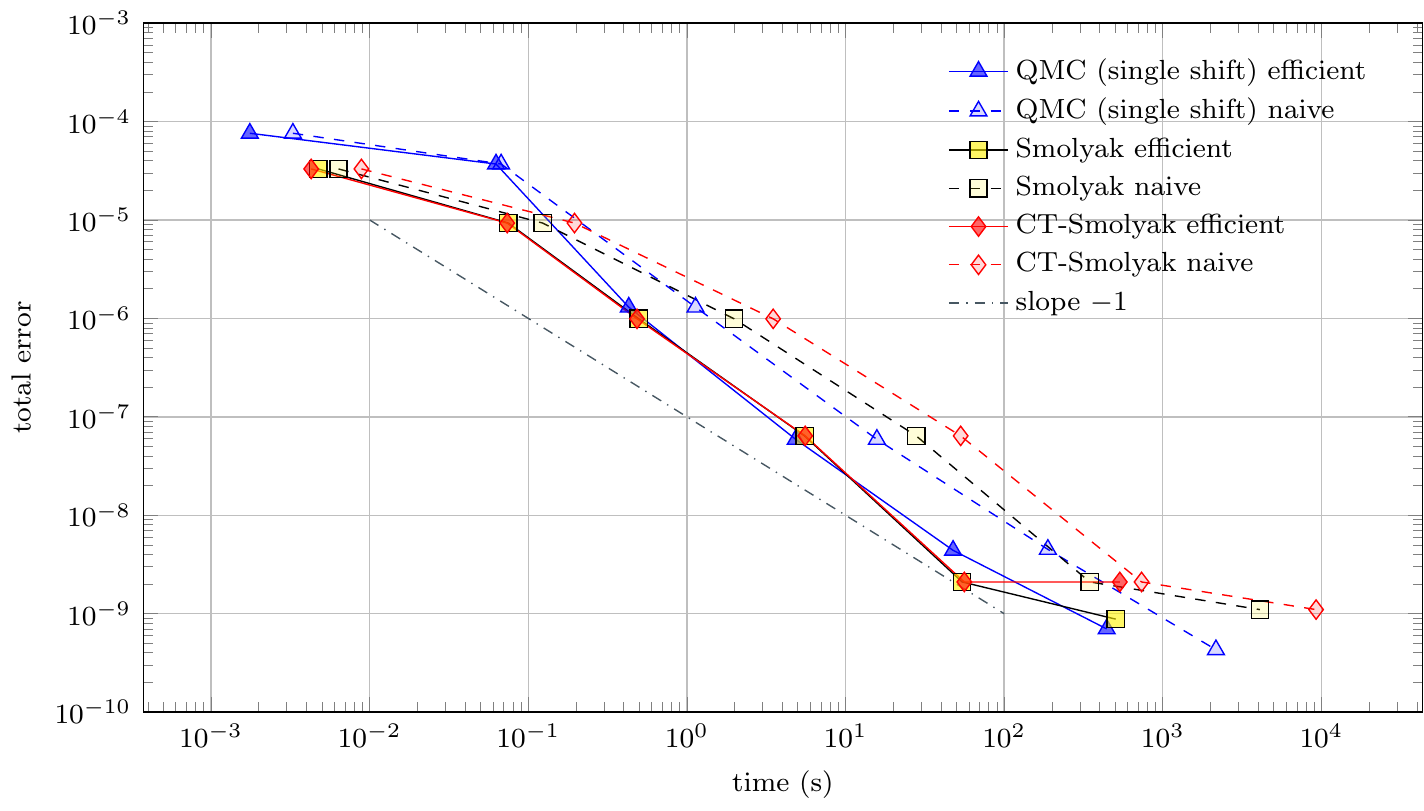}
\caption{\small Estimated total error against time.} \label{fig:speed}
\end{figure}

If we compare the direct Smolyak MDM with the combination technique
Smolyak MDM then we notice that for our efficient reformulation the two
methods have very similar running times, with a very minor advantage for
the direct method, but with the naive formulation the direct method is
the clear winner.
Note that
we can see the effect of rounding errors in the calculations for
$\e=10^{-5}$ and $10^{-6}$, since the total errors should remain the same
between the naive and efficient implementations of the same algorithm,
while the direct Smolyak and combination technique should also have the
same total errors.

In Figure~\ref{fig:speed} we plot the total error against time for the
results in Table~\ref{tab:time} to demonstrate the speedup of the
efficient implementations. For each pair of efficient and naive
results, we expect the data points to be at the same horizontal level
(same total error) but with bigger and bigger gaps in time (the speedup
factor increases) as the errors go down.


\section{Conclusion}

The MDM is a powerful algorithm for approximating integrals of
$\infty$-variate functions, but care must be taken to ensure the
implementation is efficient. In this paper we have provided details and
explicit pseudocodes explaining how to efficiently run all components of
the algorithm: from constructing the active set to running the MDM
 for both randomized QMC rules and Smolyak quadrature rules.
By reformulating the MDM we are able to save cost by reducing the amount
of repeated function evaluations incurred because of the recursive
structure of the anchored decomposition. We applied the MDM to an example
integrand that possesses similar properties to those that arise in recent
PDE problems with random coefficients. The numerical results clearly
support the cost savings of the efficient reformulations.

\subsection*{Acknowledgements}

We gratefully acknowledge the financial support from the Australian
Research Council (FT130100655 and DP150101770), the KU Leuven research
fund (OT:3E130287 and C3:3E150478), the Taiwanese National Center for
Theoretical Sciences (NCTS) -- Mathematics Division, and the Statistical
and Applied Mathematical Sciences Institute (SAMSI) 2017 Year-long Program
on Quasi-Monte Carlo and High-Dimensional Sampling Methods for Applied
Mathematics.

\vfill
\noindent\textbf{Authors' addresses}\\
\noindent Alexander D. Gilbert\\
School of Mathematics and Statistics, The University of New South Wales\\
Sydney, NSW 2052, Australia\\
\texttt{adgilbert91@gmail.com}\\

\noindent Dirk Nuyens\\
Department of Computer Science, KU Leuven\\
Celestijnenlaan 200A, 3001 Heverlee, Belgium\\
\texttt{dirk.nuyens@cs.kuleuven.be}\\

\noindent Frances Y. Kuo\\
School of Mathematics and Statistics, The University of New South Wales\\
Sydney, NSW 2052, Australia\\
\texttt{f.kuo@unsw.edu.au}\\

\noindent Grzegorz W. Wasilkowski\\
Department of Computer Science, University of Kentucky\\
301 David Marksbury Building, Lexington, KY 40506, USA\\
\texttt{greg@cs.uky.edu}\\

\end{document}